\documentclass[a4paper,english,twoside,10pt]{article}
\usepackage[a4paper,left=2.5cm,right=2.5cm, top=2.5cm, bottom=2.5cm]{geometry}
\usepackage[latin1]{inputenc}
\usepackage{cuted}
\usepackage{babel}
\usepackage[useregional]{datetime2}
\usepackage[ruled, linesnumbered]{algorithm2e}
\usepackage{amsthm}
\usepackage{mathrsfs}
\usepackage{mathtools}
\usepackage{accents}
\usepackage{multirow}
\usepackage{graphicx}
\usepackage{amscd,amsmath,amssymb,mathrsfs,bbm,listings}
\usepackage{cancel}
\usepackage{epstopdf}
\allowdisplaybreaks
\graphicspath{{figs/}}
\usepackage{amsmath}
\usepackage{footmisc}
\usepackage{amsthm}
\usepackage{amssymb}
\usepackage{stmaryrd}
\SetSymbolFont{stmry}{bold}{U}{stmry}{m}{n}
\usepackage{float}
\usepackage{bigints}
\usepackage{cite}
\usepackage{color}
\usepackage[abs]{overpic}
\usepackage[font=footnotesize,labelfont=bf]{caption}
\usepackage{footnote}
\usepackage{cases}
\usepackage{tikz}
\usepackage{rotating}
\usepackage{blkarray}
\usetikzlibrary{matrix,calc,arrows,cd}
\usepackage{pdfcomment}
\usepackage{soul,xcolor}
\usepackage{enumerate}
\usepackage{enumitem}
\usepackage{verbatim}
\usepackage{graphicx}
\usepackage{subcaption}
\usepackage{mwe}
\usepackage{tikz}

\newtheorem{theorem}{Theorem}[section]
\newtheorem{lemma}[theorem]{Lemma}

\newtheorem{corollary}[theorem]{Corollary}
\newtheorem{remark}[theorem]{Remark}

\newtheorem{assumption}{Assumption}
\newcommand{\eremk}{\hbox{}\hfill\rule{0.8ex}{0.8ex}}

\numberwithin{equation}{section}
\hypersetup{
    colorlinks,
    linktoc=section,
    citecolor=red,
    linkcolor=blue,
    urlcolor=magenta,
    pdfborder={0 0 0}
}

\newcommand\ev{\boldsymbol{\varepsilon}_{\bv}}
\newcommand\tev{\widetilde{\boldsymbol{\varepsilon}}_{\bv}}
\makeatletter
\newcount\my@repeat@count
\newcommand{\pdt}[1]{%
\partial_{
  \begingroup
  \my@repeat@count=\z@
  \@whilenum\my@repeat@count<#1\do{t\advance\my@repeat@count\@ne}%
  \endgroup}
}
\makeatother
\newcommand\epsi{\varepsilon_{\psi}}
\newcommand\tepsi{\widetilde{\varepsilon}_{\psi}}

\newcommand\epsitt{\pdt{2} \varepsilon_{\psi}}
\newcommand{\opsih}{\overline{\psi}_h}
\newcommand{\obvh}{\overline{\bv}_h}
\newcommand{\olambdah}{\overline{\lambda}_h}
\newcommand\elambda{\varepsilon_{\lambda}}
\newcommand\telambda{\widetilde{\varepsilon}_{\lambda}}
\newcommand\piM{\mathcal{P}_{\mathcal{M}}}
\newcommand\etapsih{\eta_{\psi, h}}
\newcommand\tetapsih{\widetilde{\eta}_{\psi, h}}

\newcommand\etalambdah{\eta_{\lambda, h}}
\newcommand\tetalambdah{\widetilde{\eta}_{\lambda, h}}

\newcommand\etavh{\boldsymbol{\eta}_{\bv, h}}
\newcommand\tetavh{\widetilde{\boldsymbol{\eta}}_{\bv, h}}

\newcommand\dlambda{\xi_{\lambda}}
\newcommand\tdlambda{\widetilde{\xi}_{\lambda}}
\newcommand\dpsi{\xi_{\psi}}
\newcommand\tdpsi{\widetilde{\xi}_{\psi}}
\newcommand\dbv{\boldsymbol{\xi}_{\bv}}
\newcommand\tdbv{\widetilde{\boldsymbol{\xi}}_{\bv}}
\newcommand\piHDG{\Pi_{\sf{HDG}}}
\newcommand\piS{\Pi_{\mathcal{S}}}
\newcommand\piQ{\Pi_{\boldsymbol{\mathcal{Q}}}}
\newcommand\hK{h_K}
\renewcommand{\dot}[1]{\accentset{\mbox{\bfseries .}}{#1}}
\renewcommand{\ddot}[1]{%
  \accentset{\mbox{\bfseries .\hspace{-0.25ex}.}}{#1}}

\newcommand\Sp{\mathcal{S}_h^p}

\newcommand\Mp{\mathcal{M}_h^p}
\newcommand\Qp{\boldsymbol{\mathcal{Q}}_h^p}

\newcommand\Th{\mathcal{T}_h}
\newcommand\dx{\text{d}\bx}
\newcommand\dS{\text{d}S}
\newcommand\ds{\text{d}s}
\newcommand\dt{\text{d}t}
\newcommand\Fh{\mathcal{F}_h}
\newcommand\FhI{\mathcal{F}_h^\mathcal{I}}
\newcommand\FhD{\mathcal{F}_h^{\mathcal{D}}}
\newcommand{\BK}{\mathcal{B}_{\textup{F-P}}}

\newcommand\bm{\boldsymbol{m}}

\newcommand\calS{\mathcal{S}}
\newcommand\calE{\mathcal{E}}
\newcommand\calA{\mathcal{A}}
\newcommand\calL{\mathcal{L}}
\newcommand\calR{\mathcal{R}}
\newcommand\calN{\mathcal{N}}

\newcommand\bM{\boldsymbol{M}}
\newcommand\bX{\boldsymbol{X}}

\newcommand\psih{\psi_h}
\newcommand\psihi{\psi_{h}^{(i)}}
\newcommand\tpsih{\widetilde{\psi}_h}
\newcommand\tpsiht{\pdt{1}{\widetilde{\psi}}_h}
\newcommand\psiht{\pdt{1}{\psi}_h}
\newcommand\psihtt{\pdt{2}{\psi}_h}
\newcommand\psihttt{\pdt{3}{\psi}_h}
\newcommand\Psih{\Psi_h}
\newcommand\tPsih{\widetilde{\Psi}_h}
\newcommand\V{\mathbf{V}}
\newcommand\Vh{\V_h}
\newcommand\tVh{\widetilde{\mathbf{V}}_h}

\newcommand\brh{\bunderline{\boldsymbol{r}}_h}

\newcommand\bvh{\bunderline{\boldsymbol{v}}_h}
\newcommand\bvhi{\bunderline{\boldsymbol{v}}_{h}^{(i)}}
\newcommand\tbvh{\widetilde{\bunderline{\boldsymbol{v}}}_h}
\newcommand\tbvht{\pdt{1}{\widetilde{\bunderline{\boldsymbol{v}}}}_h}

\newcommand\bvht{\pdt{1}{\bunderline{\boldsymbol{v}}}_{h}}
\newcommand\bvhtt{\pdt{2}{\bunderline{\boldsymbol{v}}}_{h}}

\newcommand\thetah{\theta_h}
\newcommand\muh{\mu_h} 
\newcommand\lambdah{\lambda_h} 
\newcommand\lambdahi{\lambda_{h}^{(i)}} 
\newcommand\tlambdah{\widetilde{\lambda}_h} 
\newcommand\tlambdaht{\pdt{1}{\widetilde{\lambda}}_h} 
\newcommand\lambdaht{\pdt{1}{\lambda}_h}
\newcommand\lambdahtt{\pdt{2}{\lambda}_h}
\newcommand\Lambdah{\Lambda_h} 
\newcommand\tLambdah{\widetilde{\Lambda}_h} 
\newcommand\wh{w_h}
\newcommand\phih{\phi_h}
\newcommand\alphah{\alpha_h}
\newcommand\alphaht{\pdt{1}{\alpha}_h}

\def\ulal{\underline{\alpha}}
\def\olal{\overline{\alpha}}
\newcommand{\I}[1]{\operatorname{I}_{#1}}

\newcommand\psiFlux{\widehat{\psi}_h}
\newlength{\dhatheight}
\newcommand\vFlux{
    \settoheight{\dhatheight}{\ensuremath{\widehat{\bv_h}}}
    \addtolength{\dhatheight}{-0.40ex}
    \widehat{\vphantom{\rule{1pt}{\dhatheight}}
    \smash{\widehat{\bv}}}_h
}

\newcommand\vFluxt{
    \settoheight{\dhatheight}{\ensuremath{\widehat{\bv_h}}}
    \addtolength{\dhatheight}{-0.40ex}
    \dpt
    \widehat{\vphantom{\rule{1pt}{\dhatheight}}
    \smash{\widehat{\bv}}}_{h}
}

\newcommand\IR{\mathbb{R}}
\newcommand\IN{\mathbb{N}}
\newcommand\QT{Q_T}
\newcommand\bv{\bunderline{\boldsymbol{v}}}
\newcommand\bupsilon{\bunderline{\boldsymbol{\Upsilon}}}

\newcommand\bGamma{{\boldsymbol{\Gamma}}}
\newcommand\bx{\boldsymbol{x}}
\newcommand\bn{\bunderline{\mathbf{n}}}

\newcommand\dpt{\partial_t}
\newcommand\dptt{\partial_{tt}}
\newcommand\dpttt{\partial_{ttt}}

\newcommand\calD{\mathcal{D}}

\newcommand{\Norm}[2]{\| #1 \|_{#2}}

\newcommand{\SemiNorm}[2]{\left| #1 \right|_{#2}}

\newcommand{\jumpN}[1]{\left\llbracket #1\right\rrbracket_{\sf{N}}}
\newcommand{\bunderline}[1]{\underaccent{\bar}{#1}}
\newcommand{\Pp}[2]{\mathbb{P}^{#1}(#2)}

\title{Asymptotic-preserving hybridizable discontinuous Galerkin method for the Westervelt quasilinear wave equation}

\author{Sergio G\'omez\thanks{Department of Mathematics and Applications, University of Milano-Bicocca, 20125 Milan, Italy (\href{mailto:sergio.gomezmacias@unimib.it}{sergio.gomezmacias@unimib.it})} \and Mostafa Meliani\thanks{Department of Mathematics, Radboud University, Heyendaalseweg 135, 6525 AJ Nijmegen,
The Netherlands (\href{mailto:mostafa.meliani@ru.nl}{mostafa.meliani@ru.nl})}}

\date{}

\begin{document}

\maketitle
\begin{abstract}
\noindent 
We discuss the asymptotic-preserving properties of a hybridizable discontinuous Galerkin method for the Westervelt model of ultrasound waves. 
More precisely, we show that the proposed method is robust with respect to small values of the sound diffusivity damping parameter~$\delta$ by deriving low- and high-order energy stability estimates, and \emph{a priori} error bounds that are independent of~$\delta$.
Such bounds are then used to show that, when~$\delta \rightarrow 0^+$, the method remains stable
and the discrete acoustic velocity potential~$\psi_h^{(\delta)}$ converges 
to~$\psi_h^{(0)}$, 
where the latter is the singular vanishing dissipation limit.
Moreover, we 
prove optimal convergence rates for the 
approximation of the acoustic particle velocity variable~$\bv = \nabla \psi$. The established theoretical results are illustrated 
with some numerical experiments.
\end{abstract}



\section{Introduction}
Let~$\QT = \Omega \times (0, T)$ be a space--time cylinder, where~$\Omega \subset \IR^d$ $(d \in \{2, 3\})$ is an open, bounded polytopic domain with Lipschitz boundary~$\partial \Omega$, and~$T > 0$ is the final time. We consider the following Westervelt equation of nonlinear acoustics~\cite{Westervelt_1963}:
\begin{equation}
\label{EQN::MODEL-PROBLEM}
\begin{cases}
(1 + 2 k\pdt{1}{\psi}) \pdt{2}{\psi} - c^2 \Delta \psi - \delta \Delta(\pdt{1}{\psi}) = 0 & \text{ in } \QT,\\
\psi = 0 & \text{ on }\partial \Omega \times (0, T),\\
\psi = \psi_0, \qquad \dpt \psi = \psi_1 & \text{ on } \Omega \times \{0\},
\end{cases}
\end{equation}
where the unknown~$\psi: \QT \rightarrow \IR$ is the acoustic velocity potential.
In the IBVP~\eqref{EQN::MODEL-PROBLEM}, the constant~$k\in \IR$ is a medium-dependent nonlinearity parameter, $c>0$ is the speed of sound, $\psi_0$ and~$\psi_1$ are given initial data, and~$\delta \geq 0$ is the sound diffusivity coefficient. 

Introducing the acoustic particle velocity variable~$\bv : \QT \rightarrow \IR^d$, defined by~$\bv := \nabla \psi$, the Westervelt equation in~\eqref{EQN::MODEL-PROBLEM}
can be rewritten in mixed form as
\begin{equation}
\label{EQN::MODEL-PROBLEM-MIXED}
\begin{cases}
(1 + 2k \pdt{1}{\psi}) \pdt{2}{\psi} - c^2 \nabla \cdot \bv - \delta \nabla \cdot (\pdt{1}\bv) = 0 & \text{ in } \QT,\\
\bv = \nabla \psi & \text{ in } \QT,\\
\psi = 0 & \text{ on } \partial \Omega \times (0, T),\\
\psi = \psi_0, \qquad \pdt{1}\psi = \psi_1 & \text{ on } \Omega \times \{0\}.
\end{cases}
\end{equation}
Since we study the limit as~$\delta\to0^+$, we make the purely technical assumption that~$\delta\in[0,\overline{\delta})$ for some fixed~$\overline{\delta} > 0$.
Such an assumption is helpful in the limiting behavior analysis in Section~\ref{SEC::ASYMPTOTIC_LIMITS}, as it allows us to make the estimates depend on~$\overline{\delta}$ but never on~$\delta$ itself.

The Westervelt equation models the propagation of sound in a fluid medium, and it is
a well-accepted model in nonlinear acoustics (see e.g.,~\cite[\S 5.3]{Kaltenbacher_2007}). Nonlinear sound propagation finds a multitude of technical and medical applications, 
such as ultrasound imaging, lithotripsy, welding, and sonochemistry; see~\cite{Demi_Verweij_2014,Kaltenbacher_ETAL_2002}.

When the parameter~$\delta$ is strictly positive, equation~\eqref{EQN::MODEL-PROBLEM} is strongly damped, and
its solution enjoys global existence properties for initial conditions satisfying some smallness and regularity  
assumptions as shown 
in~\cite{Kaltenbacher_Lasiecka_2009,Meyer_Wilke_2009}. 
Conversely, when $\delta=0$,  
the main mechanism preventing the formation of singularities is lost and no global existence results are known. 
The stark contrast between these two regimes gives rise to interesting  
issues, such as the 
continuous dependence of the solution 
on the damping parameter~$\delta \to 0^+$, and the interplay of this limit and numerical discretizations. 
A numerical method for the Westervelt equation is said to be \emph{asymptotic preserving} if it allows for interchanging the vanishing limits of the 
mesh size parameter~$h$ and the sound diffusivity parameter~$\delta$, i.e., if it satisfies the commutative diagram in Figure~\ref{FIG::ASYMPTOTIC_PRESERVING}. The main focus  
of this work is to show that the proposed method is asymptotic preserving.

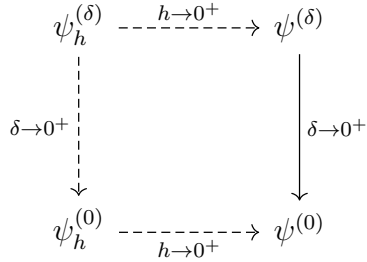
\begin{figure}[!ht]
\centering
\large{
\begin{tikzcd}[row sep = 4.5em, column sep=4.5em]
\psih^{(\delta)} \arrow[dashed, r, "h \to 0^+"] \arrow[dashed, d, "\delta \to 0^+" left]
& \psi^{(\delta)} \arrow[d, "\delta \to 0^+"] \\
\psih^{(0)} \arrow[dashed, r, "h \to 0^+" below]
& \psi^{(0)}
\end{tikzcd}
}
\caption{Asymptotic-preserving commutative diagram for the Westervelt equation. This diagram represents the connections between~$\psih^{(\delta)}$ and~$\psi^{(\delta)}$ as~$h \rightarrow 0^+$ (even in 
the limit case~$\delta = 0$) as well as between~$\psih^{(\delta)}$ and~$\psih^{(0)}$ as~$\delta \to 0^+$.  
The superscript~$(\delta)$ is used to emphasize the dependence on the parameter~$\delta$ of the continuous solution and its numerical approximation.
}\label{FIG::ASYMPTOTIC_PRESERVING}
\end{figure}
In the literature, \emph{a priori} error results for the approximation of the solution to the Westervelt equation initially relied on the assumption of strictly positive values of the damping parameter~$\delta$ (see, e.g., \cite{Nikolic_Wohlmuth_2019, Antonietti_ETAL_2020}). 
Nevertheless, as the damping parameter is relatively small in practice and it can become negligible in certain applications, 
there have been recent efforts
to devise numerical methods that are robust with respect to small values of the  
sound diffusivity parameter~$\delta$. 
In particular, estimates for the standard and mixed finite element discretizations of the Westervelt equation with~$\delta = 0$ follow as particular cases of those in~\cite{Hochbruck_Maier_2022,Maier_PhD_2020,Meliani_Nikolic_2023}, whereas the asymptotic behaviour of such methods for~$\delta \rightarrow 0^+$ has been recently studied in~\cite{Nikolic_2023,Dorich_Nikolic_2024}.  
The main challenge resides in the limited regularity offered by most standard finite element spaces, which hinders the extension of the arguments used to study the vanishing viscosity limit in the continuous setting (see, e.g., \cite{Kaltenbacher_Nikolic_2022}). 

This work concerns the asymptotic analysis of a hybridizable discontinuous Galerkin~(HDG) method for the Westervelt equation when~$\delta \to 0^+$. 
HDG methods, originally introduced in~\cite{Cockburn_Gopalakrishnan_Lazarov_2009} for an elliptic PDE, are a class of discontinuous Galerkin methods characterized by the possibility of performing a local \emph{static condensation} procedure to reduce the number of unknowns of the linear system stemming from the discretization of a $d$-dimensional linear PDE. 
Such a procedure leads to linear systems involving only unknowns associated 
with degrees of freedom on~$(d-1)$-dimensional mesh-facets. 
Although this hybridization property does not naturally extend to nonlinear PDEs, it can be used in combination with suitable nonlinear solvers (see, e.g., Section~\ref{SECT::FULLY-DISCRETE} below). 
Moreover, provided that the exact solution is smooth enough, the Local Discontinuous Galerkin-hybridizible (LDG-H) method in~\cite{Cockburn_Dong_Guzman_2008,Cockburn_Gopalakrishnan_Lazarov_2009} for the Poisson equation converges with optimal order~$\mathcal{O}(h^{p + 1})$ for the~$L^2(\Omega)$-error of the flux variable when approximations of degree~$p$ are used, and allows for a local postprocessing that produces an approximation of degree~$p + 1$ of the primal variable that superconverges with order~$\mathcal{O}(h^{p + 2})$ in the~$L^2(\Omega)$-norm.

To the best of our knowledge, there are four different versions of the HDG method for the linear acoustic wave equation~$(c^{-2} \dptt u - \Delta u = f)$: 
\begin{enumerate}[label = \emph{\alph*)}, topsep 
= 2pt, itemsep = 1.4pt]
\item the dissipative HDG method introduced in~\cite{Nguyen_Peraire_Cockburn_2011} and analyzed in~\cite{Cockburn_Quenneville-belair_2014}, which is based on the first-order system~$(\dpt \bunderline{\boldsymbol{q}} = \nabla v;\  c^{-2} \dpt v - \nabla \cdot \bunderline{\boldsymbol{q}} = f)$ with~$v := \dpt u$ and~$\bunderline{\boldsymbol{q}} := \nabla u$; 
\item the conservative HDG method in~\cite{Griesmaier_Monk_2014} based on the same first-order system, whose energy conserving property is enforced by choosing the \emph{numerical fluxes} 
of~$\bunderline{\boldsymbol{q}}_h$ in dependence of~$\dpt v_h$, which in turn causes a theoretical loss of convergence of half an order;  
\item the HDG method in~\cite{Sanchez_Ciuca_Peraire_Cockburn_2017} for the Hamiltonian formulation~$(\dpt u = v;\ c^{-2} \dpt v = f + \nabla \cdot \bunderline{\boldsymbol{q}})$;
and
\item the conservative HDG method in~\cite{Cockburn-ETAL-2018}, which is based on the 
mixed formulation~$(\bunderline{\boldsymbol{q}} = -\nabla u; \ c^{-2}\dptt u + \nabla \cdot \bunderline{q} = f)$.
\end{enumerate}
The theoretical results in \emph{a)}, \emph{c)}, and \emph{d)} predict optimal convergence for the approximation of all the variables involved, and superconvergence for some (locally computable) postprocessed approximations of the scalar variables. 

In this work, we design an HDG method 
for the Westervelt model, which is based on the conservative HDG method in~\cite{Cockburn-ETAL-2018} for the linear second-order wave equation. This choice allows us to directly approximate the variables of interest~$(\psi, \bv)$, eliminate efficiently the discrete vector variable~$\bvh$ from the nonlinear ODE system, and obtain optimal convergence in the low- and high-order energy norms. Moreover, it facilitates the extension of the techniques used in~\cite{Meliani_Nikolic_2023} for the analysis of mixed FEM discretizations of the Westervelt equation.

\paragraph{Main contributions.}
The main theoretical results in this work are as follows: under some sensible assumptions on the smallness and regularity of the exact solution, we show that
\begin{enumerate}[label = \emph{\roman*)}, topsep 
= 2pt, itemsep = 1.4pt]
\item There exists a unique solution to the proposed HDG semidiscrete formulation.
\item Optimal convergence rates of order~$\mathcal{O}(h^{p+1})$ 
are achieved for the error of the method in some energy norms. 
In particular, the higher accuracy obtained for the approximation of the acoustic particle velocity~$\bv$ exceeds the one expected for standard DG discretizations; cf. \cite{Antonietti_ETAL_2020}. 
An accurate numerical approximation of~$\bv$ is relevant, e.g., for enforcing absorbing conditions~\cite{shevchenko2015absorbing} or gradient-based shape optimization of focused ultrasound devices~\cite{kaltenbacher2016shape,meliani2022analysis}.
\item The method is asymptotic preserving (i.e., the commutative diagram in Figure~\ref{FIG::ASYMPTOTIC_PRESERVING} holds), which implies that 
the semidiscrete approximation does not degenerate when~$\delta \to 0^+$.
\end{enumerate}

These theoretical results are validated in Section~\ref{SEC::NUMERICAL_EXPERIMENTS} below by some numerical examples.
In addition, we numerically observe superconvergence of the discrete approximation of~$\psi$ obtained by the local postprocessing technique in~\cite[Eq.~(2.2)]{Cockburn-ETAL-2018}.


\paragraph*{Outline of the paper.}  
In Section~\ref{SEC::DISCRETIZATION}, we introduce the discrete spaces and the HDG 
semidiscrete formulation for model~\eqref{EQN::MODEL-PROBLEM-MIXED}.
In Section~\ref{SECT::LINEARIZED-PROBLEM}, we study the well-posedness and derive \emph{a priori} error estimates for an auxiliary linearized problem.
By means of a fixed-point argument, such results are extended in Section~\ref{SEC::NONLINEAR_WESTERVELT} to the nonlinear Westervelt equation. Section~\ref{SEC::ASYMPTOTIC_LIMITS} is devoted to establishing the convergence of the numerical scheme to its vanishing $\delta$-limit. In Section~\ref{SEC::NUMERICAL_EXPERIMENTS}, we describe a fully discrete scheme obtained by combining the proposed HDG method with a predictor-corrector Newmark time discretization, and illustrate our theoretical findings with some numerical experiments.
We end this work with some concluding
remarks in Section~\ref{SECT::CONCLUSIONS}.

\paragraph*{Notation.} 
We denote the first, second, and third partial derivatives with respect to the time variable~$t$ of a function~$v$ by~$\pdt{1}{v}$, $\pdt{2}{v}$, and~$\pdt{3}{v}$, respectively.

We shall use the notation $x \lesssim y$, which stands for $x \leq Cy$, where $C$ is a generic constant that does not depend on the mesh size parameter~$h$ nor on the sound
diffusivity parameter~$\delta$.

Standard notation for $L^p$, Sobolev, and Bochner 
spaces is employed throughout. For example, for a given bounded, Lipschitz domain~$D \subset \IR^d$ ($d \in \IN$) and~$s \in \IR^+$, the Sobolev space~$H^s(D)$ is endowed with the standard inner product~$(\cdot, \cdot)_{s, D}$, the seminorm~$|\cdot|_{H^s(D)}$, and the norm~$\|\cdot\|_{H^s(D)}$.
In particular, for~$s = 0$, the space~$H^0(D) := L^2(D)$ is the space of 
Lebesgue square integrable functions over~$D$, and we simply denote its standard 
inner product by~$(\cdot, \cdot)_{D}$.

Let~$n\in \mathbb{N}$, $p\in[1,\infty]$, and~$X$ be a Banach space, and denote by $\partial_t^i$ the $i$th
partial derivative with respect to time. The Bochner space 
$$W^{n,p}(0,T;X) := \{u \in L^p(0,T;X),\quad \partial_t^i u\in L^p(0,T;X) \quad  \forall i\leq n\}$$ 
is endowed with the norm 
$$\|u\|_{W^{n,p}(0,T;X)} := \sum_{i=0}^n \|\partial_t^i u\|_{L^p(0,T;X)} \qquad \text{for all } u\in W^{n,p}(0,T;X).$$

\section{Semidiscrete HDG formulation}\label{SEC::DISCRETIZATION}
Let~$\{\Th\}_{h > 0}$ be a family of conforming simplicial meshes for the domain~$\Omega$ satisfying the standard shape-regularity and quasi-uniformity conditions. 
We denote by~$\Fh = \FhI \cup \FhD$ the set of mesh facets of~$\Th$, where~$\FhI$ and~$\FhD$ are the sets of internal and Dirichlet boundary facets, respectively. 
For each element~$K \in \Th$, we denote by~$(\partial K)^{\mathcal{I}}$ 
and~$(\partial K)^{\calD}$ the union of the facets of~$K$ that belong to~$\FhI$ 
and~$\FhD$, respectively. Denoting the diameter of each element~$K$ by~$\hK$, we 
define the mesh size~$h:= \max_{K \in \Th} \hK$.

Given~$p \in \IN$, we define the following piecewise polynomial spaces:
\begin{equation}
\label{EQN::DG-SPACES}
\Sp := \prod_{K \in \Th} \Pp{p}{K}, \qquad \Qp := \prod_{K \in \Th} \Pp{p}{K}^d, \qquad \Mp := \prod_{F \in \FhI} \Pp{p}{F},
\end{equation}
where~$\Pp{p}{K}$ and~$\Pp{p}{F}$ denote the spaces of polynomials of total degree at most~$p$ on~$K$ and~$F$, respectively. 
We denote by~$\jumpN{\cdot}$ the normal jump operator, which is defined for all~$\wh \in \Sp$ and~$\brh \in \Qp$ as
\begin{align*}
\begin{cases}
\begin{tabular}{ll} 
$\jumpN{\wh} := \wh {}_{|_{K_1}} \bn_{K_1} + \wh {}_{|_{K_2}} \bn_{K_2}$ & \multirow{2}{*}{ on~$F = \partial K_1 \cap \partial K_2 \in \FhI, \text{ for some } K_1, K_2 \in \Th$,} \\
$\jumpN{\brh} := \brh {}_{|_{K_1}} \cdot \bn_{K_1} + \brh {}_{|_{K_2}} \cdot \bn_{K_2}$
\end{tabular}
\end{cases}
\end{align*}
where~$\bn_{K}$ denotes the outward-pointing unit normal vector on~$\partial K$.
For any positive real number~$s$, we define the following broken Sobolev space:
\begin{equation*}
H^s(\Th) := \{v \in L^2(\Omega) \ : \ v_{|_{K}} \in H^s(K) \ \forall K \in \Th\}.
\end{equation*}

The proposed hybridizable discontinuous Galerkin semidiscrete formulation for the Westervelt equation in~\eqref{EQN::MODEL-PROBLEM-MIXED} is\footnote{In this work, the vector variable~$\bvh$ approximates~$\nabla \psi$, whereas it typically approximates~$-\nabla \psi$ in elliptic problems. 
As a consequence, there are some slight differences in the standard HDG tools used in the coming sections.}: 
for all~$t \in (0, T]$, find~$(\psih(\cdot, t), \bvh(\cdot, t), \lambdah(\cdot, t)) \in \Sp \times \Qp \times \Mp$ 
such that the following equations are satisfied for all~$K \in \Th$:
\begin{subequations}
\label{EQN::SEMI-DISCRETE}
\begin{align}
\int_K \bvh \cdot \brh \dx  = \int_{\partial K} \psiFlux \brh \cdot \bn_K \dS - \int_K \psih \nabla \cdot \brh \dx & \qquad \hspace*{\fill} \forall \brh \in \Qp, \\
\begin{split}
\int_K (1 + 2 k \psiht) \psihtt \wh \dx - \int_{\partial K} \wh (c^2 \vFlux + \delta \vFluxt)\cdot \bn_K \dS \\
+ \int_K (c^2 \bvh + \delta \bvht) \cdot \nabla \wh \dx = 0 & \qquad \hspace*{\fill} \forall \wh \in \Sp,
\end{split}
\end{align}
the following compatibility equation is satisfied for all~$F \in \FhI$:
\begin{equation}
\label{EQN::HDG-3}
\int_{F} \muh \jumpN{\vFlux} \dS  = 0 \qquad \forall \muh \in \Mp,
\end{equation}
and appropriate discrete initial conditions, which will be specified in Section~\ref{SEC::INITIAL_DATA_CHOICE}, are prescribed.
\end{subequations}

The \emph{numerical fluxes}~$\psiFlux$ and~$\vFlux$ are approximations of the traces of~$\psih$ and~$\bvh$ on~$\Fh$, and are defined as follows (see~\cite[\S 3.2]{Cockburn_Guzman_Wang_2009}):
\begin{equation}
\label{EQN::NUMERICAL-FLUXES}
\psiFlux := 
\begin{cases}
\lambdah & \text{ if } F \in \FhI, \\
0 & \text{ if } F \in \FhD, 
\end{cases}
\qquad 
\vFlux := 
\begin{cases}
\bvh - \tau (\psih - \lambdah) \bn_K & \text{ if } F \in \FhI, \\
\bvh - \tau \psih \bn_{\Omega}  & \text{ if } F \in \FhD,
\end{cases}
\end{equation}
for some piecewise constant function~$\tau$ that is double valued on~$\FhI$ and single valued on~$\FhD$. In particular, we consider the single-facet choice introduced in~\cite[Eq.~(1.6)]{Cockburn_Dong_Guzman_2008}, i.e., given a strictly positive constant~$\bar{\tau}$, we define~$\tau$ on each element~$K \in \Th$ as
\begin{equation}
\label{EQN::SINGLE-FACET}
\tau{}_{|_{\partial K}} := \begin{cases}
0 & \text{ on } \partial K \backslash F_K^{\tau}, \\
\bar{\tau} & \text{ on } F_K^{\tau},
\end{cases}
\end{equation}
for a fixed facet~$F_K^{\tau}$ of~$K$.
The compatibility condition~\eqref{EQN::HDG-3} implies that the normal component of~$\vFlux$ is single valued on the mesh skeleton, i.e., $\jumpN{\vFlux} = 0$ on~$\FhI$. 

We define the following inner products:
\begin{equation*}
(u, v)_{\Th} := \sum_{K \in \Th} (u, v)_{K}, \qquad
(u, v)_{\partial \Th} := \sum_{K \in \Th} (u, v)_{\partial K}, \ \qquad 
(u, v)_{(\partial \Th)^{\mathcal{I}}} := \sum_{K \in \Th} (u, v)_{(\partial 
K)^{\mathcal{I}}}.
\end{equation*}

Given bases for the spaces in~\eqref{EQN::DG-SPACES}, let~$M$, $\bM$, $B$, $S$, $E$, $F$, and~$G$ be the matrix representations of the following bilinear forms:\footnote{These bilinear forms are also well defined for sufficiently regular functions. }
\begin{align*}
m_h(\psih, \wh) & := ( \psih,  \wh)_{\Th}  & & \forall \psih, \wh \in \Sp,\\
\bm_h(\bvh, \brh) & := (\bvh, \brh)_{\Th} & & \forall \bvh, \brh \in \Qp, \\
b_h(\psih, \brh) &  := (\psih, \nabla \cdot \brh)_{\Th}  & & \forall (\psih, \brh) \in \Sp \times \Qp,\\
s_h(\psih, \wh) & := (\tau \psih, \wh)_{\partial \Th} & & \forall \psih, \wh \in \Sp, \\
e_h(\lambdah, \brh) & := - (\lambdah, \jumpN{\brh})_{\FhI} & & \forall (\lambdah, \brh) \in \Mp \times \Qp, \\
f_h(\lambdah, \wh) & := - (\tau \lambdah, \wh)_{(\partial \Th)^{\mathcal{I}}} & & \forall (\lambdah, \wh) \in \Mp \times \Sp,\\
g_h(\lambdah, \muh) & := (\tau \lambdah, \muh)_{(\partial \Th)^{\mathcal{I}}} & & \forall \lambdah, \muh \in \Mp,
\end{align*}
and~$\calN_h(\cdot, \cdot)$ be the vector representation of the nonlinear operator
\begin{equation*}
n_h(\phih; \thetah, \wh) := \sum_{K \in \Th} \int_K (1 + 2k \phih) \thetah \wh \dx \qquad \hspace{\fill} \forall \phih, \thetah, \wh \in \Sp.
\end{equation*}
Then, after summing up over all the elements~$K \in \Th$, replacing the numerical fluxes by their definition in~\eqref{EQN::NUMERICAL-FLUXES}, and using the following notation:
\begin{equation}
\label{EQN::TILDA-NOTATION}
\tlambdah = \lambdah + \frac{\delta}{c^2} \lambdaht, \quad \tpsih = \psih + \frac{\delta}{c^2} \psiht, \quad \text{and} \quad \tbvh = \bvh + \frac{\delta}{c^2} \bvht,
\end{equation}
the semidiscrete HDG formulation~\eqref{EQN::SEMI-DISCRETE} can be written in 
operator form as follows: for all~$t \in (0, T]$, find~$(\psih(\cdot, t), 
\bvh(\cdot, t), \lambdah(\cdot, t)) \in \Sp \times \Qp \times \Mp$ 
such that
\begin{subequations}
\label{EQN::SEMI-DISCRETE-OPERATORS}
\begin{align}
\label{EQN::SEMI-DISCRETE-OPERATORS-1}
\bm_h(\bvh, \brh) + b_h(\psih, \brh) + e_h(\lambdah, \brh)  & = 0 & & \forall \brh \in \Qp, \\
\label{EQN::SEMI-DISCRETE-OPERATORS-2}
n_h(\psiht, \psihtt, \wh) 
- c^2b_h(\wh, \tbvh) + c^2 s_h(\tpsih, \wh) + c^2 f_h(\tlambdah, \wh) & = 0 &  &  \forall \wh \in \Sp, \\
\label{EQN::SEMI-DISCRETE-OPERATORS-3}
-e_h(\muh, \bvh) + f_h(\muh, \psih) + g_h(\lambdah, \muh) & = 0 & &  \forall \muh \in \Mp,
\end{align}
\end{subequations}
which leads to the following system of nonlinear ordinary differential equations (ODEs):
\begin{subequations}
\begin{align*}
\bM \Vh + B \Psih + E \Lambdah & = 0, \\
\begin{split}
\calN_h\left(\frac{d}{dt}{\Psih},  \frac{d^2}{dt^2}\Psih\right) 
-c^2 B^T \tVh 
+ c^2 S\tPsih + c^2 F \tLambdah & = 0,
\end{split}
\\
-E^T \Vh + F^T \Psih +  G \Lambdah & = 0.
\end{align*}
\end{subequations}

\begin{remark}[Structure of~$\calN_h(\cdot, \cdot)$]
\label{REM::STRUCTURE-N}
Since the nonlinear operator~$\calN_h(\cdot, \cdot)$ is linear with respect to its second argument, 
it can also be written as~$\calN_h(\frac{d}{dt} \Psih, \frac{d^2}{dt^2}\Psih) = N_h(\frac{d}{dt}\Psih) \frac{d^2}{dt^2}\Psih$, for some block diagonal matrix~$N_h = N_h(\frac{d}{dt}\Psih)$.
\eremk
\end{remark}

\begin{remark}[Linear case]
Setting~$\delta = 0$ and~$k = 0$ in  
the semidiscrete 
formulation~\eqref{EQN::SEMI-DISCRETE-OPERATORS}, the conservative HDG method in~\cite{Cockburn-ETAL-2018} for the linear acoustic wave equation is recovered.
\eremk
\end{remark}

\section{Linearized semidiscrete HDG formulation~\label{SECT::LINEARIZED-PROBLEM}}
As an intermediate step for the asymptotic and convergence analysis of the semidiscrete HDG formulation~\eqref{EQN::SEMI-DISCRETE-OPERATORS} for the Westervelt equation, 
we analyze an auxiliary linearized problem with damping parameter~$\delta \geq 0$ and a variable coefficient. We first make some assumptions on the data of the linearized problem. 
In Section~\ref{SUBSECT::WELL-POSEDNESS}, we show some low- and high-order energy stability estimates and discuss the existence of a unique semidiscrete solution.
In Section~\ref{SECT::ERROR-SEMI-DISCRETE-LINEARIZED}, we show some \emph{a priori} error bounds in the energy norms. 
The choice of the discrete initial conditions is discussed in Section~\ref{SEC::INITIAL_DATA_CHOICE}. Optimal~$h$-convergence rates for the error in the energy norms are proven in Section~\ref{SECT::h-convergence}.

We consider the following auxiliary, potentially damped, perturbed linear wave equation:
\begin{equation}
\label{EQN::CONTINUOUS-LINEARIZED}
\begin{cases}
(1 + 2k\alpha
) \pdt{2}{\psi} - c^2 \nabla \cdot \bv - \delta\nabla \cdot (\pdt{1}\bv) = \varphi & \text{ in } \QT,\\
\bv = \nabla \psi + \bupsilon & \text{ in }\QT, \\
\psi = 0 & \text{ on } \partial \Omega \times (0, T), \\
\psi = \psi_0, \qquad \pdt{1}\psi = \psi_1 & \text{ on } \Omega \times \{0\},
\end{cases}
\end{equation}
for some given functions~$\varphi : \QT \rightarrow \IR$, $ \alpha{: \QT \rightarrow \IR}$, and $\bupsilon{:\QT \rightarrow \IR^d}$. 
The force term~$\varphi$ will be used to represent the consistency error due to the approximation of~$\dpt \psi$~by $\alpha$.
The perturbation function~$\bupsilon$ will be used in 
Theorem~\ref{THM::ENERGY-STABILITY} to 
represent the error resulting from the low-order~$L^2{(\Omega)}$-orthogonality 
properties of the HDG projection {in~\eqref{EQN::HDG_PROJECTION}} of~$\bv$. This 
will be useful in proving the error bounds of 
Theorem~\ref{THM:ERROR_BOUND_THEOREM}; see the system of error 
equations~\eqref{EQN::ERROR-ETA-EQUATIONS}.
Such an error term also appears in the  
analysis of the HDG method for the linear acoustic wave equation in~\cite[Lemma 3.1]{Cockburn_Quenneville-belair_2014}.

We consider the following semidiscrete HDG formulation for the auxiliary problem in~\eqref{EQN::CONTINUOUS-LINEARIZED}: for all~$t \in (0, T]$, find~$(\psih(\cdot, t), \bvh(\cdot, t), \lambdah(\cdot, t)) \in \Sp \times \Qp \times \Mp$ such that
\begin{subequations}
\label{EQN::SEMI-DISCRETE-LINEARIZED}
\begin{align}
\label{EQN::SEMI-DISCRETE-LINEARIZED-1}
\bm_h(\bvh, \brh) + b_h(\psih, \brh) + e_h(\lambdah, \brh)  & = (\bupsilon, \brh)_{\Omega} & & \forall \brh \in \Qp, \\
\nonumber
m_h((1 + 2k\alpha_h) \psihtt, \wh)  - c^2b_h(\wh, \tbvh) + c^2 s_h(\tpsih, \wh) & \\
\label{EQN::SEMI-DISCRETE-LINEARIZED-2}
   + c^2 f_h(\tlambdah, \wh)  & = (\varphi, \wh)_{\Omega} & & \forall \wh \in \Sp, 
\\
\label{EQN::SEMI-DISCRETE-LINEARIZED-3}
- e_h(\muh, \bvh) + f_h(\muh, \psih) + g_h(\lambdah, \muh) & = 0 & &  \forall \muh \in \Mp,
\end{align}
\end{subequations}
where~$\alphah$ is a discrete approximation of~$\alpha$.
To complete the system of differential equations~\eqref{EQN::SEMI-DISCRETE-LINEARIZED}, it is necessary to compute appropriate discrete initial conditions from the initial data of the continuous problem~$\psi_0$, $\psi_1$. 
A suitable choice for these initial conditions is essential in the error analysis below. We discuss our choice for the discrete initial conditions in Section~\ref{SEC::INITIAL_DATA_CHOICE}.

To show the well-posedness of the
semidiscrete problem~\eqref{EQN::SEMI-DISCRETE-LINEARIZED}, we make the following assumptions on 
the semidiscrete coefficient~$\alpha_h$, the forcing function~$\varphi$, and the perturbation function~$\bupsilon$.
\begin{assumption}\label{Assumption_nondeg}
Let $T>0$. We assume that~$\varphi \in H^{1}(0, T; L^2(\Omega))$, $\bupsilon \in W^{3, 1}(0,T; L^2(\Omega)^d)$, and the coefficient $\alpha_h \in H^{1}(0,T; \Sp)$ is non degenerate, i.e., there exist constants $\ulal$, $\olal>0$ independent of~$h$ and~$\delta$, such that
\begin{equation}
\label{EQN::NON-DEGENERACY-ALPHA}
	0 < 1-2|{k}|\ulal \leq 1+2{k}\alpha_h(\bx,t) \leq 1+2|{k}|\olal \quad  \forall (\bx,t) \in \Omega \times (0,T).
\end{equation}

Furthermore, we assume that there exist constants~$0< \gamma_0 < {\sigma_0} < 1$ independent of $h$ and the damping parameter~$\delta$ such that 
\begin{equation} 
\label{EQN::ENERGY-ASSUMPTION}
\frac{|k|}{1 - 2|k|\ulal} \Norm{\alphaht}{L^1(0, T; L^{\infty}(\Omega))}+ \frac{\gamma_0}{2} \leq \frac{{\sigma_0}}{2}.
\end{equation}
\end{assumption}

\begin{remark}[Linearization argument]
It is fairly common in the (numerical) analysis of quasilinear wave equations to combine a linearized problem with nondegeneracy assumptions on the variable coefficient.
Such assumptions are then shown to be verified by the solution to the nonlinear problem
by using a fixed-point strategy; 
see Theorem~\ref{THM::FIXED-POINT} below. See also~\cite[Thm.~3]{Antonietti_ETAL_2020}, \cite[Thm.~6.1]{Nikolic_Wohlmuth_2019}, and \cite[Thm.~4.1]{Meliani_Nikolic_2023} for similar arguments. 
\eremk
\end{remark}

\subsection{Well-posedness and energy 
estimates \label{SUBSECT::WELL-POSEDNESS}}
In this section, we discuss the existence and uniqueness of the solution to the semidiscrete formulation~\eqref{EQN::SEMI-DISCRETE-LINEARIZED}, and derive some low- and high-order energy stability estimates.

We first write the semidiscrete formulation~\eqref{EQN::SEMI-DISCRETE-LINEARIZED} in matrix form as
\begin{subequations}
\label{EQN::LIN-ODE}
\begin{align}
\label{EQN::LIN-ODE-1}
\bM \Vh + B \Psih + E \Lambdah & = \bGamma, \\
\label{EQN::LIN-ODE-2}
\begin{split}
{N_h}(\alpha_h)  \frac{d^2}{dt^2}\Psih
-c^2 B^T \tVh 
+ c^2 S\tPsih + c^2 F \tLambdah  & = \Phi,
\end{split}
\\
\label{EQN::LIN-ODE-3}
-E^T \Vh + F^T \Psih +  G \Lambdah & = 0,
\end{align}
\end{subequations}
where~$\Phi$ and~$\boldsymbol\Gamma$ are{, respectively,} the vector representations of the terms in~\eqref{EQN::SEMI-DISCRETE-LINEARIZED-1} and~\eqref{EQN::SEMI-DISCRETE-LINEARIZED-2} involving~$\phi$ and~$\bupsilon$. 
The matrix~${N_h} = {N_h}(\alphah)$, defined in Remark~\ref{REM::STRUCTURE-N}, is symmetric positive definite on account of the nondegeneracy assumption made in~\eqref{EQN::NON-DEGENERACY-ALPHA} on~$\alphah$. 
From~\eqref{EQN::LIN-ODE-1} and~\eqref{EQN::LIN-ODE-3}, we deduce that
\begin{equation}
\label{EQN::LINEAR-RELATION-VARIABLES}
\begin{pmatrix}
\bM & E \\
-E^T & G
\end{pmatrix}
\begin{pmatrix}
\tVh \\
\tLambdah
\end{pmatrix}
= \begin{pmatrix}
-B \tPsih + \widetilde{\bGamma}\\
-F^T \tPsih
\end{pmatrix}.
\end{equation} 
Since~$\bM$ and~$G$ are symmetric positive definite matrices, the block matrix on the left-hand side of~\eqref{EQN::LINEAR-RELATION-VARIABLES} is nonsingular. 
Therefore, $\tVh$ and~$\tLambdah$ can be expressed in terms of~$\tPsih$ and {$\widetilde{\bGamma}$} through~\eqref{EQN::LINEAR-RELATION-VARIABLES}. 
This implies that 
the ODE system~\eqref{EQN::LIN-ODE} can be reduced to a second-order linear ODE system 
involving only~$\Psih$ by multiplying equation~\eqref{EQN::LIN-ODE-2} by the matrix~${N_h}(\alphah)^{-1}$.  
{If} Assumption~\ref{Assumption_nondeg} holds, 
classical ODE theory~(see, e.g., \cite[Thm.~1.8]{Agarwal_Oregan_2001}) predicts 
the existence of a unique solution~$\psih \in W^{3,1}(0,T; \Sp)$. 
Moreover, through \eqref{EQN::LIN-ODE-1} and~\eqref{EQN::LIN-ODE-3}, we obtain that $\bvh \in W^{3,1}(0,T; \Qp)$ and~$\lambdah \in W^{3, 1}(0, T; \Mp)$. 
In the analysis below, the embedding~$W^{3, 1}(0, T) \hookrightarrow C^{2}([0, T])$ is of utmost relevance.

We derive low- and high-order energy stability estimates for the semidiscrete formulation~\eqref{EQN::SEMI-DISCRETE-LINEARIZED}.

\begin{theorem}[Energy estimates for the discrete linearized problem] \label{THM::ENERGY-STABILITY} 
Let~$T>0$, $c > 0$, and~$\delta \geq 0$.
Assume that the semidiscrete-in-space coefficient~$\alphah$, the 
forcing function~$\varphi$, and the perturbation function~$\bupsilon$ satisfy Assumption~\ref{Assumption_nondeg}.
Then, the solution to 
semidiscrete formulation~\eqref{EQN::SEMI-DISCRETE-LINEARIZED} satisfies the following energy stability estimates:
\begin{subequations}
\begin{align}
\nonumber
\sup_{t \in (0, T)} \calE_h^{(0)}[\psih, \bvh, \lambdah](t) \leq (1 - \sigma_0)^{-1} & \Big( \calE_h^{(0)}[\psih, \bvh, \lambdah](0) + \frac{{T}}{2\gamma_0(1 - 2|k|\ulal)}\Norm{\varphi}{L^2(0, T; L^2(\Omega))}^2 \\
\label{EQN::ENERGY-STABILITY-LINEARIZED-1}
& + {\Big(\frac{\delta}{4}  + {\frac{c^2T}{2\sigma_0}} \Big)}\Norm{\dpt \bupsilon}{L^2(0, T;L^2(\Omega)^d)}^2\Big),\\
\nonumber
\sup_{t \in (0, T)} \calE_h^{(1)}[\psih, \bvh, \lambdah](t) \leq (1 - {\sigma_0})^{-1} & \Big(\calE_h^{(1)}[\psih, \bvh, \lambdah](0) + \frac{{T}}{2\gamma_0(1 - 2|k|\ulal)}\Norm{\dpt \varphi}{L^2(0, T; L^2(\Omega))}^2 \\
\label{EQN::ENERGY-STABILITY-LINEARIZED-2}
& + {\Big(\frac{\delta}{4} + \frac{c^2 T}{2{\sigma_0}}\Big)} \Norm{\dptt \bupsilon}{L^2(0, T;L^2(\Omega)^d)}^2\Big),
\end{align}
\end{subequations}
where~$\sigma_0$ is the constant in the smallness assumption~\eqref{EQN::ENERGY-ASSUMPTION}, and the discrete energy functionals~$\calE_h^{(0)}{[\cdot, \cdot, \cdot]}(t)$ and $\calE_h^{(1)}{[\cdot, \cdot, \cdot]}(t)$ are given by
\begin{align*}
\calE_h^{(0)}[\psih, \bvh, \lambdah](t) := & \frac12 \Norm{\sqrt{1+{2}k\alphah} \psiht}{L^2(\Omega)}^2 \\
& + \frac{c^2}{2} \left(\Norm{\bvh}{L^2(\Omega)^d}^2 + \Norm{\tau^{\frac12} (\lambdah - \psih)}{L^2\left((\partial \Th)^{\mathcal{I}}\right)}^2 + \Norm{\tau^{\frac12} \psih}{L^2\left((\partial \Th)^{\calD}\right)}^2\right),\\
\calE_h^{(1)}[\psih, \bvh, \lambdah](t) := & \frac12 \Norm{\sqrt{1+2k\alphah} \psihtt}{L^2(\Omega)}^2 \\
& + \frac{c^2}{2} \left(\Norm{\bvht}{L^2(\Omega)^d}^2 + \Norm{\tau^{\frac12} (\lambdaht - \psiht)}{L^2\left((\partial \Th)^{\mathcal{I}}\right)}^2 + \Norm{\tau^{\frac12} \psiht}{L^2\left((\partial \Th)^{\calD}\right)}^2\right).
\end{align*}
\end{theorem}
\begin{proof}

The proofs of the energy estimates in~\eqref{EQN::ENERGY-STABILITY-LINEARIZED-1} and~\eqref{EQN::ENERGY-STABILITY-LINEARIZED-2} are postponed to Appendices~\ref{APP:ProofLowerOrder} and~\ref{APP:ProofHigherOrder}, respectively.
\end{proof}

\begin{remark}[Regularity of~$\bupsilon$]
As can be seen from estimates~\eqref{EQN::ENERGY-STABILITY-LINEARIZED-1} and~\eqref{EQN::ENERGY-STABILITY-LINEARIZED-2}, it is sufficient to have $\bupsilon\in H^2(0,T;L^2(\Omega)^d)$. 
However, this would degrade the regularity to be expected from the solution to the semidiscrete problem~\eqref{EQN::LIN-ODE}. In particular, we would only get that~$\bvh \in H^2(0,T; \Qp)$ and~$\lambdah \in H^2(0, T; \Mp)$. Since $\bupsilon$ is only an auxiliary function used to 
represent the error introduced by the low-order~$L^2(\Omega)$-orthogonality of the HDG projection used in the error analysis (see Theorem~\ref{THM:ERROR_BOUND_THEOREM} below), we 
assume~$\bupsilon \in W^{3,1}(0,T;L^2(\Omega)^d)$, thus retaining the expected regularity of~$\bvh$ and~$\lambdah$ when~$\bupsilon = 0$ as in the original problem~\eqref{EQN::SEMI-DISCRETE-OPERATORS}. 
\end{remark}

Estimates~\eqref{EQN::ENERGY-STABILITY-LINEARIZED-1} and~\eqref{EQN::ENERGY-STABILITY-LINEARIZED-2} show boundedness of the energy of the semidiscrete solution with respect to  
the initial energies, the forcing 
{function}~$\varphi$, and {the perturbation function}~$\bupsilon$. 
In order to show that these constitute indeed stability results, we need to show that the initial discrete energies, $\calE_h^{(0)}{[\psih, \bvh, \lambdah]}(0)$ and~$\calE_h^{(1)}{[\psih, \bvh, \lambdah]}(0)$, remain bounded uniformly in~$h$. We prove the stability result for the nonlinear problem in Lemma~\ref{LEMMA::ENERGY-STABILITY}.

\begin{remark}[Stabilization parameter]
In order to obtain the energy stability estimates {in}~\eqref{EQN::ENERGY-STABILITY-LINEARIZED-1} and~\eqref{EQN::ENERGY-STABILITY-LINEARIZED-2}, 
we only require the stabilization parameter~$\bar{\tau}$ 
in~\eqref{EQN::SINGLE-FACET} to be strictly positive.
Moreover, there are no polynomial inverse estimates involved in the proof of Theorem~\ref{THM::ENERGY-STABILITY}.
\eremk
\end{remark}

\subsection{\emph{A priori} error estimates \label{SECT::ERROR-SEMI-DISCRETE-LINEARIZED}}
In this section, we carry out an~\emph{a priori} error analysis for the semidiscrete formulation~\eqref{EQN::SEMI-DISCRETE-LINEARIZED}.
To do so, we first recall the properties of
some special HDG projections. 
For all~$\epsilon > 0$, let~$\piM : H^{\frac12 + \epsilon}(\Th) \rightarrow \Mp$ be the~$L^2$-orthogonal projection in~$\Mp$, defined for all~$u \in H^{\frac12 + \epsilon}(\Th)$ as
\begin{equation}\label{EQN::FACET_PROJECTION}
(\piM u  - u, \muh)_{{(\partial \Th)^{\mathcal{I}}}} = 0 \qquad \forall \muh \in \Mp,
\end{equation}
and let~$\piHDG := (\piS, \piQ) : H^{\frac12 + \epsilon}(\Th) \times H^{\frac12 + \epsilon}(\Th)^d \rightarrow \Sp \times \Qp$ 
be the HDG projection in~\cite[Eq.~(2.1)]{Cockburn_Gopalakrishnan_Sayas_2010},  defined for all~$(\psi, \bv) \in H^{\frac12 + \epsilon}(\Th) \times H^{\frac12 + \epsilon}(\Th)^d $ and all~$K \in \Th$ as
\begin{subequations}\label{EQN::HDG_PROJECTION}
\begin{align}
(\piQ \bv - \bv, \brh)_{K} & = 0 \qquad \qquad \forall \brh \in \Pp{p-1}{K}^d, \\
(\piS \psi - \psi, \wh)_{K} & = 0 \qquad \qquad \forall \wh \in \Pp{p-1}{K}, \\
\left((\widehat{\piQ \bv} - \bv) \cdot \bn_{K}, \muh\right)_{f} & = 0 \qquad \qquad \forall \text{ facets } F \subset \partial K,\  \forall \muh \in \Pp{p}{F},
\end{align}
\end{subequations}
where
\begin{equation*}
\widehat{\piQ \bv} \cdot \bn_K := \piQ \bv \cdot \bn_K - \tau (\piS \psi - \piM \psi) \qquad \text{ on }\partial K. 
\end{equation*}

Let~$(\psi, \bv)$ be the solution to the continuous Westervelt equation in~\eqref{EQN::MODEL-PROBLEM-MIXED}, and let~$(\psih, \bvh)$ be the solution to the semidiscrete formulation~\eqref{EQN::SEMI-DISCRETE-LINEARIZED} for the linearized problem~\eqref{EQN::CONTINUOUS-LINEARIZED} with~$\bupsilon = 0$ and~$\varphi = 0$. 

We define the following error functions:
\begin{subequations}
\begin{align}
\label{EQN::ERROR-FUNCTIONS}
& \epsi := \psi - \psih, & & \ev := \bv - \bvh, & & \elambda := \psi - \lambdah, \\
\label{EQN::PROJECTION-ERRORS}
& \dpsi := \piS \psi - \psi, & &\dbv := \piQ \bv - \bv, & & \dlambda := \piM \psi - \psi, \\
\label{EQN::DISCRETE-ERRORS}
& \etapsih := \piS \psi - \psi_h,  & & \etavh  := \piQ \bv - \bv_h, & & \etalambdah := \piM \psi - \lambdah,
\end{align}
\end{subequations}
and recall the approximation properties of~$\piHDG$ in~\cite[{Thm.~2.1}]{Cockburn_Gopalakrishnan_Sayas_2010}.

\begin{lemma}[Approximation properties of~$\piHDG$]\label{LEMMA::Projections_Errors}
Suppose~$p \geq 0$, $\tau_{|_{\partial K}}$ is nonnegative, and~$\tau_K^{\max} := \max \tau_{|_{\partial K}} > 0$. Then, $\piHDG(\psi, \bv) = (\piS \psi, \piQ \bv)$ is well defined. 
Furthermore, there is a constant~$C_\Pi > 0 $ independent of~$K$ and~$\tau$ such that
\begin{subequations}
\begin{align*}
\Norm{\dbv}{L^2(K)} & \leq C_{\Pi} \left(\hK^{s_{\bv} + 1} \SemiNorm{\bv}{H^{s_{\bv} + 1}(K)^{{d}}} + \hK^{s_{\psi} + 1} \tau_K^{\star} \SemiNorm{\psi}{H^{s_{\psi} + 1}(K)}\right), 
\\
\Norm{\dpsi}{L^2(K)} & \leq C_{\Pi} \Big(\hK^{s_{\psi} + 1} \SemiNorm{\psi}{H^{s_{\psi} + 1}(K)} + \frac{\hK^{s_{\bv} + 1}}{\tau_K^{\max}} \SemiNorm{\nabla \cdot \bv}{H^{s_{\bv}}(K)} \Big), 
\end{align*}
\end{subequations}
for~$s_{\psi}, s_{\bv} \in [0, p]$ {and~$(\psi, \bv) \in H^{s_{\psi} + 1}(K) \times H^{s_{\bv} + 1}(K)^d$}.  
Above, $\tau_K^{\star} := \max \tau_{|_{\partial K \backslash F^{\star}}}$, where~$F^{\star}$ is a 
{facet} of~$K$ at which~$\tau_{|_{\partial K}}$ is maximum.
\end{lemma}
For the single-facet choice in~\eqref{EQN::SINGLE-FACET}, we have {that}~$\tau_K^\star = 0$ and~$\tau_K^{\max} = \bar{\tau}$\, for all~$K \in \Th$. In particular, the error bound for~$\dbv$ does not depend on the regularity of~$\psi$.

The following lemma is crucial for the error analysis of HDG methods.
\begin{lemma} {For all~$(\psi, \bv) \in H^{\frac12 + \epsilon}(\Th) \times H^{1}(\Th)^d$, it holds}
\begin{equation}
\label{EQN::COMMUTATIVITY-IDENTITY}
b_h(\wh, \dbv) = s_h(\wh, \dpsi) \qquad \forall \wh \in \Sp.
\end{equation}
\end{lemma}
\begin{proof}
This identity is an immediate consequence of the weak commutativity property in~\cite[Prop.~2.1]{Cockburn_Gopalakrishnan_Sayas_2010}.
\end{proof}
By the consistency of the proposed method and recalling the tilde~$(\sim)$ notation from~\eqref{EQN::TILDA-NOTATION}, the following error equations are verified:
\begin{subequations}
\begin{align*}
\bm_h(\ev, \brh) + b_h(\epsi, \brh) + e_h(\elambda, \brh)  & = 0 & & \forall \brh \in \Qp, \\
\nonumber
m_h((1 + 2k\alpha_h) \epsitt, \wh)  - c^2b_h(\wh, \tev) &\\+ c^2 s_h(\tepsi, \wh)
   + c^2 f_h(\telambda, \wh)  & = -  m_h( 2k({\dpt \psi}-\alphah) \pdt{2}\psi, \wh) & & \forall \wh \in \Sp, 
\\
- e_h(\muh, \ev) + f_h(\muh, \epsi) + g_h(\elambda, \muh) & = 0 & & \forall \muh \in \Mp.
\end{align*}
\end{subequations}

We are in a position to obtain~\emph{a priori} error bounds for the semidiscrete linearized formulation~\eqref{EQN::SEMI-DISCRETE-LINEARIZED} with respect to the continuous solution to the Westervelt equation in~\eqref{EQN::MODEL-PROBLEM-MIXED}.
\begin{theorem}[Error bounds for the semidiscrete linearized  
formulation]\label{THM:ERROR_BOUND_THEOREM}
Under the assumptions of Theorem~\ref{THM::ENERGY-STABILITY}, the following error bounds are satisfied:
\begin{subequations}
\begin{align}
\nonumber
\sup_{t \in (0, T)} & \Big(\frac{c^2}{2}\Norm{\ev}{L^2(\Omega)^d}^2 + \frac12 \Norm{\sqrt{1 + 2k\alphah}\dpt \epsi}{L^2(\Omega)}^2 \Big) \\
\nonumber 
\leq & \sup_{t \in (0, T)} \big(c^2\Norm{\dbv}{L^2(\Omega)^d}^2 + \Norm{\sqrt{1 + 2k \alphah}\dpt \dpsi}{L^2(\Omega)}^2 \big) + 2(1 - {\sigma_0})^{-1}\Big(\calE_h^{(0)}[\etapsih, \etavh, \etalambdah](0)\\
\label{EQN::ERROR_BOUND-LINEARIZED-1}
& + \frac{{T}}{2 \gamma_0(1 - {2}|k|\ulal)} \Norm{{\hat{\varphi}}}{L^2(0, T; L^2(\Omega))}^2 
+ \Big({\frac{\delta}{4}
+ \frac{c^2T}{2{\sigma_0}}} \Big) \Norm{\dpt \dbv}{L^2(0, T; L^2(\Omega)^d)}^2\Big), \\
\nonumber
\sup_{t \in (0, T)} & \Big(\frac{c^2}{2}\Norm{\dpt \ev}{L^2(\Omega)^d}^2 + \frac12 \Norm{\sqrt{1 + 2k\alphah}\dptt \epsi}{L^2(\Omega)}^2 \Big) \\
\nonumber
\leq & \sup_{t \in (0, T)} \Big(c^2\Norm{\dpt \dbv}{L^2(\Omega)^d}^2 + \Norm{\sqrt{1 + 2k \alphah}\dptt \dpsi}{L^2(\Omega)}^2\Big) + 2(1 - {\sigma_0})^{-1} \Big(\calE_h^{(1)}[\etapsih, \etavh, \etalambdah](0)\\
\label{EQN::ERROR_BOUND-LINEARIZED-2}
& + \frac{{T}}{2 \gamma_0(1 - 2|k|\ulal)} \Norm{{\dpt \hat\varphi}}{L^2(0, T; L^2(\Omega))}^2 
+ \left({\frac{\delta}{4} 
+ \frac{c^2 T}{2{\sigma_0}}} \right)\Norm{\dptt \dbv}{L^2(0, T; L^2(\Omega)^d)}^2\Big),
\end{align}
\end{subequations}
where~${\hat{\varphi}} \in 
H^1(0,T; \Sp)$ is given by
\begin{equation}
\label{EQN::DEF-VARPHI}
{\hat{\varphi} =  \Pi_0\left[(1 + 2k\alphah) {\pdt{2}\dpsi} + 2k(\dpt\psi-\alphah) \pdt{2}\psi\right],}
\end{equation}
with~$\Pi_0$ denoting the~$L^2(\Omega)$-orthogonal projection in~$\Sp$.
\end{theorem}
\begin{proof}
We only present the proof of the error bound in~\eqref{EQN::ERROR_BOUND-LINEARIZED-1}, as the proof of~\eqref{EQN::ERROR_BOUND-LINEARIZED-2} is similar.

We split the error functions in~\eqref{EQN::ERROR-FUNCTIONS} as
\begin{equation*}
\epsi = \etapsih - \dpsi, \qquad \ev = \etavh - \dbv, \qquad \elambda = \etalambdah - \dlambda.
\end{equation*}
The definition of the HDG projections in~\eqref{EQN::FACET_PROJECTION} and~\eqref{EQN::HDG_PROJECTION} implies that, for all~$t \in (0, T]$, the discrete error functions~$(\etapsih(\cdot, t), \etavh(\cdot, ), \etalambdah(\cdot, t)) \in \Sp \times \Qp \times \Mp$ solve a semidiscrete 
linearized problem as in~\eqref{EQN::SEMI-DISCRETE-LINEARIZED}. More precisely, they satisfy the following equations for all~$(\wh, \brh, \muh) \in \Sp \times \Qp \times \Mp$:
\begin{subequations}\label{EQN::ERROR-ETA-EQUATIONS}
\begin{align}
\label{EQN::ERROR-ETA-EQUATIONS-1}
    \bm_h(\etavh, \brh) + b_h(\etapsih, \brh) + e_h(\etalambdah, \brh) &= -
    (\dbv, \brh)_{\Omega} \\
\label{EQN::ERROR-ETA-EQUATIONS-2}
m_h((1 + 2k\alphah) \pdt{2}{\etapsih}, \wh)  - c^2b_h(\wh, \tetavh) + c^2 s_h(\tetapsih, \wh) + c^2 f_h(\tetalambdah, \wh) 
	 &= (\hat\varphi,\wh)_{\Omega},  \\
\label{EQN::ERROR-ETA-EQUATIONS-3}
- e_h(\muh, \etavh) + f_h(\muh, \etapsih) + g_h(\etalambdah, \muh) &= 0,
\end{align}
\end{subequations}
where~$\hat{\varphi} \in {H^1}
(0, T; \Sp)$ is a lifting function defined by the following projection:
\begin{equation*}
\begin{split}
(\hat{\varphi}, \wh)_{\Omega}  := &\ m_h((1 + 2k\alpha_h) {\pdt{2}\dpsi}, \wh) {+}  m_h( 2k({\dpt \psi}-\alphah) \pdt{2}\psi, \wh)  \\
& \qquad - c^2b_h(\wh, \tdbv) + c^2 s_h(\tdpsi, \wh) 
 + c^2 f_h(\tdlambda, \wh) \qquad \forall \wh \in \Sp.
\end{split}
\end{equation*}

From the definition of~$\piM$ in~\eqref{EQN::FACET_PROJECTION} and identity~\eqref{EQN::COMMUTATIVITY-IDENTITY}, we deduce that
\begin{equation*}
f_h(\tdlambda, \wh) = 0 \qquad \text{and} \qquad - b_h(\wh, \tdbv) + s_h(\tdpsi, \wh) = 0 \qquad \forall \wh \in \Sp,
\end{equation*}
which implies that~$\hat{\varphi}$ satisfies~\eqref{EQN::DEF-VARPHI}.

The desired bound is then obtained from the triangle inequality and the energy estimate~\eqref{EQN::ENERGY-STABILITY-LINEARIZED-1} in Theorem~\ref{THM::ENERGY-STABILITY}.
\end{proof}

\subsection{Choice of the discrete initial conditions}\label{SEC::INITIAL_DATA_CHOICE}
All the results presented so far are 
valid for any choice of the discrete initial conditions.
However, in order to show optimal 
convergence rates for the error in the low- and high-order energy norms,
we assume that~$\psi_0, \psi_1 \in H^2(\Omega) \cap H_0^1(\Omega)$ and choose the discrete initial conditions~$\psihi$ $(i = {0, 1})$ as the solution to the following discrete HDG elliptic problem: find~$(\psihi, \bvhi, \lambdahi) \in \Sp \times \Qp \times \Mp$ such that
\begin{subequations}
\label{EQN::DISCRETE-INITIAL-CONDITION}
\begin{align}
\bm_h(\bvhi, \brh) + b_h(\psihi, \brh) + e_h(\lambdahi, \brh) & = 0 & &  \forall \brh \in \Qp, \\
-b_h(\wh, \bvhi)  + s_h(\psihi, \wh) + f_h(\lambdahi, \wh) & = (-\Delta \psi_i, \wh)_{\Th} & & \forall \wh \in \Sp, \label{EQN::DISCRETE_INITIAL_DATA_2ND} \\
-e_h(\muh, \bvhi) + f_h(\muh, \psihi) + g_h(\lambdahi, \muh) & = 0 & & \forall \muh \in \Mp.
\end{align}
\end{subequations}
This choice of the discrete initial conditions can be interpreted as an HDG variant of the well-known Ritz projection, which was used in the numerical analysis for the strongly damped Westervelt equation in~\cite{Nikolic_Wohlmuth_2019}.

The variational problem~\eqref{EQN::DISCRETE-INITIAL-CONDITION} corresponds to 
the HDG discretization of a Poisson problem with homogeneous Dirichlet boundary 
conditions and a source term given by~$-\Delta \psi_i$.
Therefore, the existence and uniquess of a solution 
to~\eqref{EQN::DISCRETE-INITIAL-CONDITION} follows from 
\cite[Thm.~2.3]{Cockburn_Dong_Guzman_2008}.

In next lemma, we provide bounds for the terms containing the discrete errors~$(\etapsih, \etavh, \etalambdah)$ on the right-hand side of the \emph{a priori} bounds~\eqref{EQN::ERROR_BOUND-LINEARIZED-1} and~\eqref{EQN::ERROR_BOUND-LINEARIZED-2}.

\begin{lemma}[Estimates at~$t = 0$]\label{LEMMA::ERROR_INITIAL_DATA}
Assume that~$\psi_0, \psi_1 \in H^2(\Th) \cap H_0^1(\Omega)$, and the discrete initial conditions are chosen as in~\eqref{EQN::DISCRETE-INITIAL-CONDITION}. 
Then, the following bounds hold:
\begin{subequations}
\begin{align}
\label{EQN::INITIAL-LOW-ORDER-BOUND}
\calE_h^{(0)}[\etapsih, \etavh, \etalambdah](0) & \leq \frac{(1 + 2|k| 
\olal)}{2} \Norm{\piS \psi_1 - \psih^{(1)}}{L^2(\Omega)}^2
+ \frac{c^2}{2} \Norm{\dbv(\cdot, 0)}{L^2(\Omega)^d}^2, \\
\nonumber
\calE_h^{(1)}[\etapsih, \etavh, \etalambdah](0) & \leq \frac{c^2}{2} \Norm{\dpt 
\dbv(\cdot, 0)}{L^2(\Omega)^d}^2 + \frac{(1 + 2|k| \olal)^2}{(1 - 2|k| 
\ulal)}\Norm{\dptt\dpsi(\cdot, 0)}{L^2(\Omega)}^2 \\
\label{EQN::INITIAL-HIGH-ORDER-BOUND}
& \quad + \frac{4k^2}{1 - 2|k| \ulal}\Norm{(\pdt1 \psi - \alphah)(\cdot, 0) 
\dptt \psi(\cdot, 0)}{L^2(\Omega)}^2.
\end{align}
\end{subequations}
Moreover, if the domain~$\Omega$ is such that 
\begin{equation}
\label{EQN::REGULARITY-OMEGA}
\varphi \in H_0^1(\Omega), \ \  \Delta \varphi \in L^2(\Omega) \ \Longrightarrow \ \varphi \in H^2(\Omega),
\end{equation}
then, there exists a constant~$C_* > 0$ independent of~$h$ and~$\delta$ such that
\begin{equation}
\label{EQN::PSI1-ESTIMATE}
\Norm{\piS \psi_1 - \psih^{(1)}}{L^2(\Omega)} \leq C_* h \Norm{\dpt \dbv (\cdot, 0)}{L^2(\Omega)^d}.
\end{equation}
\end{lemma}
\begin{proof}
By using the nondegeneracy assumption in~\eqref{EQN::NON-DEGENERACY-ALPHA}, the low-order bound in~\eqref{EQN::INITIAL-LOW-ORDER-BOUND} can be proven as in~\cite[Lemma~3.6]{Cockburn-ETAL-2018} {for the linear wave equation}, {whereas estimate~\eqref{EQN::PSI1-ESTIMATE} follows from~\cite[Thm.~4.1]{Cockburn_Gopalakrishnan_Sayas_2010}}.
In contrast to~\cite[Lemma~3.6]{Cockburn-ETAL-2018}, due to the choice of~$\psih^{(1)}$ in~\eqref{EQN::DISCRETE-INITIAL-CONDITION}, the term~$\dpt \etapsih(\cdot, 0) = \piS \psi_1 - \psih^{(1)}$ does not vanish.

As for bound~\eqref{EQN::INITIAL-HIGH-ORDER-BOUND}, proceeding again as in~\cite[Lemma~3.6]{Cockburn-ETAL-2018}, we get
\begin{equation}
\label{EQN::AUX-HIGH-ORDER-INITIAL-BOUND}
\calE_h^{(1)}[\etapsih, \etavh, \etalambdah](0) \leq \frac12 \Norm{\sqrt{1 + 2k \alphah(\cdot, 0)} \dptt \etapsih (\cdot, 0)}{L^2(\Omega)}^2 + \frac{c^2}{2}\Norm{\dpt \dbv(\cdot, 0)}{L^2(\Omega)^d}.
\end{equation}

Hence, it only remains to bound the first term on the right-hand side of~\eqref{EQN::AUX-HIGH-ORDER-INITIAL-BOUND}. 
To do so, we choose~$\wh = \dptt \etapsih (\cdot, 0)$ in~\eqref{EQN::ERROR-ETA-EQUATIONS-2} for~$t = 0$  (the explicit evaluation at~$t = 0$ is omitted in the subsequent steps),
which leads to the following identity:
\begin{align*}
\Norm{\sqrt{1 + 2k \alphah} \dptt \etapsih}{L^2(\Omega)}^2 
= & \ c^2 \big(b_h(\dptt \etapsih, \etavh) 
- s_h(\etapsih, \dptt \etapsih) 
- f_h(\etalambdah, \dptt \etapsih) \big) \\
& + \delta \big(b_h(\dptt \etapsih, \dpt \etavh) 
- s_h(\dpt \etapsih, \dptt \etapsih) 
- f_h(\dpt \etalambdah, \dptt \etapsih) \big) \\
&  
+ (\hat\varphi, \dptt \etapsih)_{\Omega},
\end{align*}
where~$\hat\varphi \in \Sp(\Th)$ is defined in~\eqref{EQN::DEF-VARPHI}.

The choice of the discrete initial conditions~$\psihi$ ($i = {0, 1}$) in~\eqref{EQN::DISCRETE-INITIAL-CONDITION}, the definition of~$\piM$ in~\eqref{EQN::FACET_PROJECTION}, and identity~\eqref{EQN::COMMUTATIVITY-IDENTITY} imply that 
\begin{align*}
b_h(\dptt \etapsih, \etavh) 
& - s_h(\etapsih, \dptt \etapsih) 
- f_h(\etalambdah, \dptt \etapsih) 
= 0,\\
b_h(\dptt \etapsih, \dpt \etavh) 
& - s_h(\dpt \etapsih, \dptt \etapsih) 
- f_h(\dpt \etalambdah, \dptt \etapsih) 
= 0.
\end{align*}
Therefore, using the Cauchy--Schwarz inequality and the stability of the~$L^2(\Omega)$-orthogonal projection~$\Pi_0$, we get
\begin{equation*}
\begin{split}
& \frac12 \Norm{\sqrt{1 + 2k \alphah} \dptt \etapsih}{L^2(\Omega)}^2 \leq \frac12 \Norm{(1 + 2k \alphah)^{-\frac12}\hat\varphi}{L^2(\Omega)}^2 \\
& \qquad \leq (1 - 2|k| \ulal)^{-1} \big((1 + 2|k| 
\olal)^2\Norm{\dptt\dpsi(\cdot, 0)}{L^2(\Omega)}^2 + 4k^2 \Norm{(\dpt \psi - 
\alphah)(\cdot, 0) \dptt \psi(\cdot, 0)}{L^2(\Omega)}^2 \big),
\end{split}
\end{equation*}
which, together with {bound}~\eqref{EQN::AUX-HIGH-ORDER-INITIAL-BOUND}, completes the proof. 
\end{proof}

\subsection{\texorpdfstring{$h$}{h}-convergence\label{SECT::h-convergence}}
In order to obtain optimal $h$-convergence rates in Theorem~\ref{THM::ERROR-ESTIMATES} below for the error in the low- and high-order energy norms, we will assume that the nonlinear Westervelt equation in~\eqref{EQN::MODEL-PROBLEM-MIXED} has a regular enough solution. 
We refer the reader to~\cite{Kaltenbacher_Nikolic_2022,Kaltenbacher_Meliani_Nikolic_2023} for~$\delta$-uniform analyses of the Westervelt equation. 
Higher-order regularity of the exact solution follows from~\cite[Thm.~2.2]{Kawashima_1992} under stronger regularity and smallness assumptions on the initial conditions, and higher-order compatibility of the initial and boundary data.

Henceforth, we assume that~$h < 1$. We will also make the following assumption on how well the semidiscrete coefficient~$\alphah$ approximates~$\dpt \psi$. 
This assumption will later be verified by means of a fixed-point argument. 

\begin{assumption}\label{Assumption_approx} 
For given~$s_{\psi},s_{\bv}  \in [0, p]$, we assume that the {semidiscrete} coefficient~$\alphah$ and its time derivative~$\dpt \alphah$ approximate~$\dpt \psi$ and~$\dptt \psi$, respectively, up to the following accuracy:
\begin{equation*}
    \begin{split}
	\|{\dpt \psi} -\alphah\|_{L^\infty(0,t; L^2(\Omega))} & \leq C_* \big(h^{s_{\psi}+1} \|\psi\|_{H^{2}(0, {t}; H^{s_{\psi} + 1}(\Omega))} + h^{s_{\bv}+1} \Norm{\bv}{H^{{2}}(0, {t}; H^{s_{\bv} + 1}(\Omega)^d)}\big),\\
	\|{\dptt \psi }-\alphaht\|_{L^2(0,t; L^2(\Omega))} &  \leq C_*  \big(h^{s_{\psi}+1} \|\psi\|_{H^{3}(0, {t}; H^{s_{\psi} + 1}(\Omega))} + h^{s_{\bv}+1} \Norm{\bv}{H^{3}(0, {t}; H^{s_{\bv} + 1}(\Omega)^d)}\big),
    \end{split}
\end{equation*}
for all~$t\in [0,T]$, where the constant~$C_*>0$ does not depend on~$h$ or~$\delta$.

To establish the higher-order-in-time error estimate in~\eqref{EQN::ERROR_LINEAR_HIGH} below, we 
make a uniform boundedness assumption on the time derivative of the linear coefficient $\alphah$, namely, we require that
\begin{equation} 
\label{EQN::ERROR-HIGH-ASSUMPTION}\Norm{\alphaht}{L^2(0, T; L^{\infty}(\Omega))}\leq \check\alpha,
\end{equation}
for some positive constant~$\check\alpha$ 
independent of~$h$ and~$\delta$.

\end{assumption}
The smallness assumption in \eqref{EQN::ERROR-HIGH-ASSUMPTION} matches the one made in~\cite[Assumpt.~W1]{Meliani_Nikolic_2023} 
for the analysis of the mixed FEM approximation of the Westervelt equation.
\begin{theorem}[Error estimate for the semidiscrete linearized problem]\label{THM::ERROR-ESTIMATES}
Let~$h\in(0, \overline{h})$ and let the assumptions of Theorem~\ref{THM::ENERGY-STABILITY} and Assumption~\ref{Assumption_approx} hold. 
Let additionally~$\psi \in H^{3}(0, T; H_0^1(\Omega)\cap H^{s_{\psi} + 1}(\Omega))$ for some~$s_{\psi} \in [0, p]$ 
and~$\bv \in H^3(0, T; H^{s_{\bv} + 1}(\Omega)^d)$ for some~$s_{\bv} \in [0, p]$ be the solution to the IBVP for the Westervelt equation in~\eqref{EQN::MODEL-PROBLEM-MIXED}.
Let also~$\Omega$ be such that the regularity condition in~\eqref{EQN::REGULARITY-OMEGA} holds, and the discrete initial condition be chosen as in Section~\ref{SEC::INITIAL_DATA_CHOICE}. Then, 
\begin{subequations}
\begin{align}
\nonumber
   \sup_{t \in (0, T)}&  \big( \Norm{\ev}{L^2(\Omega)^d}^2 + \Norm{\dpt \epsi}{L^2(\Omega)}^2\big)  \\ &\lesssim \big(h^{2s_{\psi}+2} \|\psi\|^2_{H^{2}(0, {T}; H^{s_{\psi} + 1}(\Omega))} + h^{2s_{\bv}+2} \Norm{\bv}{H^{2}(0, {T}; H^{s_{\bv} + 1}(\Omega)^d)}^2\big)\big(1 + \Norm{\dptt \psi}{L^2(0, T; L^\infty(\Omega))}^2\big),
\end{align}
and
\begin{align}
   \nonumber
   \sup_{t \in (0, T)}&  \big( \Norm{\dpt\ev}{L^2(\Omega)^d}^2 + \Norm{\pdt2 \epsi}{L^2(\Omega)}^2\big)  \lesssim \big(h^{2s_{\psi}+2} \|\psi\|^2_{H^{3}(0, {T}; H^{s_{\psi} + 1}(\Omega))} \\ 
   \label{EQN::ERROR_LINEAR_HIGH}
   & + h^{2s_{\bv}+2} \Norm{\bv}{H^{3}(0, {T}; H^{s_{\bv} + 1}(\Omega)^d)}^2\big)\big(1 + \Norm{\dptt \psi}{L^\infty(0, T; L^\infty(\Omega))}^2+\Norm{\dpttt \psi}{L^2(0, T; L^\infty(\Omega))}^2\big),
   \end{align}
\end{subequations}
where the hidden constants are independent of~$h$ and~$\delta$.
\end{theorem}
\begin{proof}
We start from the estimates in Theorem~\ref{THM:ERROR_BOUND_THEOREM}. We then combine them with Lemma~\ref{LEMMA::ERROR_INITIAL_DATA}, the H\"older inequality, and the approximation properties in Lemma~\ref{LEMMA::Projections_Errors} of the HDG projection.
Furthermore, the terms involving the forcing
function~$\hat{\varphi}$ in~\eqref{EQN::DEF-VARPHI} are estimated using the Cauchy--Schwarz and the H\"older inequalities as follows:
\begin{align*}
\Norm{\hat{\varphi}}{L^2(0, T; L^2(\Omega))} \leq & \ \Norm{1 + 2k \alphah}{L^{\infty}(0, T; L^{\infty}(\Omega))}\Norm{\dptt \dpsi}{L^2(0, T; L^2(\Omega))} \\
&+ 2|k|\Norm{\dpt \psi - \alphah}{L^{\infty}(0, T; L^2(\Omega))} \Norm{\dptt \psi}{L^2(0, T; L^{\infty}(\Omega))},\\
\Norm{\dpt \hat{\varphi}}{L^2(0, T; L^2(\Omega))} \leq & 
\ \Norm{\dpt \alphah}{L^2(0, T; L^{\infty}(\Omega))} \Norm{\dptt \dpsi}{L^{\infty}(0, T; L^2(\Omega))} \\
& + \Norm{1 + 2k \alphah}{L^{\infty}(0, T; L^{\infty}(\Omega))}\Norm{\dpttt \dpsi}{L^2(0, T; L^2(\Omega))},\\
& + \Norm{\dptt \psi - \dpt \psih}{L^2(0, T; L^2(\Omega))} \Norm{\dptt \psi}{L^{\infty}(0, T; L^2(\Omega))} \\
&  + \Norm{\dpt \psi - \alphah}{L^{\infty}(0, T; L^2(\Omega))} \Norm{\dpttt \psi}{L^2(0, T; L^{\infty}(\Omega))}.
\end{align*}
Finally, the terms involving the semidiscrete coefficient~$\alphah$ can be bounded using Assumption~\ref{Assumption_approx}.

The following estimates are then obtained:
\begin{subequations}
\begin{align*}
\nonumber
\sup_{t \in (0, T)} & \Big( \Norm{\ev}{L^2(\Omega)^d}^2 + \Norm{\dpt \epsi}{L^2(\Omega)}^2\Big) \\
\nonumber
\lesssim &
\ h^{2s_{\bv} + 2} \Big(
\SemiNorm{\bv(\cdot, 0)}{H^{s_{\bv} + 1}(\Omega)^d}^2 + h \SemiNorm{\dpt \bv(\cdot, 0)}{H^{s_{\bv} + 1}(\Omega)^d}^2 + \SemiNorm{\nabla \cdot (\dptt \bv)}{L^2(0, T; H^{s_{\bv}}(\Omega))}^2  \\
\nonumber
& \qquad \qquad + \sup_{t \in (0, T)} \SemiNorm{\bv}{H^{s_{\bv} + 1}(\Omega)^d}^2 + \sup_{t \in (0, T)} \SemiNorm{\nabla \cdot (\dpt\bv)}{H^{s_{\bv}}(\Omega)}^2 + \SemiNorm{\dpt \bv}{L^2(0, T; H^{s_{\bv} + 1}(\Omega)^d)}^2 \\
\nonumber
& \qquad \qquad
+ \Norm{\bv}{H^2(0, T; H^{s_{\bv} + 1}(\Omega)^d)}^2\big(1 + \Norm{\dptt \psi}{L^2(0, T; L^\infty(\Omega))}^2\big) \Big) \\
\nonumber
& + h^{2s_{\psi} + 2} \Big( \sup_{t \in (0, T)}\SemiNorm{\dpt \psi}{H^{s_{\psi} + 1}(\Omega)}^2+ \SemiNorm{\dptt \psi}{L^2(0, T; H^{s_\psi + 1}(\Omega))}^2 + \Norm{\dptt \psi}{L^2(0, T; H^{s_{\psi} + 1}(\Omega))}^2
\\
& \qquad \qquad \ + \Norm{\psi}{H^2(0, T; H^{s_{\psi} + 1}(\Omega))}^2 \Norm{\dptt \psi}{L^2(0, T; L^\infty(\Omega))}^2
\Big), 
\end{align*}
\vspace{-0.2in}
\begin{align*}
\nonumber
\sup_{t \in (0, T)} & \Big(\Norm{\dpt \ev}{L^2(\Omega)^d}^2 + \Norm{\dptt \epsi}{L^2(\Omega)}^2\Big) \\
\nonumber 
\lesssim & \ h^{2s_{\bv} + 2} \Big(
\SemiNorm{\nabla \cdot (\dptt \bv)(\cdot, 0)}{H^{s_{\bv}}(\Omega)}^2 + \SemiNorm{\dpt \bv(\cdot, 0)}{H^{s_{\bv} + 1}(\Omega)^d}^2 + \sup_{t \in (0, T)} \SemiNorm{\dpt \bv}{H^{s_{\bv} + 1}(\Omega)^d}^2 \\
\nonumber
& \qquad \quad + \sup_{t \in (0, T)} \SemiNorm{\nabla \cdot (\dptt \bv)}{H^{s_{\bv}}(\Omega)}^2 + \SemiNorm{\dptt \bv}{L^2(0, T; H^{s_{\bv} + 1}(\Omega)^d)}^2 + \Norm{\nabla \cdot (\dptt \bv)}{L^\infty(0, T; H^{s_{\bv}}(\Omega)^d)}^2\\
\nonumber 
& \qquad \quad  +  \Norm{\nabla \cdot (\dpttt \bv)}{L^2(0, T; H^{s_{\bv}}(\Omega)^d)}^2  + \Norm{\bv}{H^{3}(0, T; H^{s_{\bv} + 1}(\Omega)^d)}^2\Norm{\dpttt \psi }{L^2(0, T; L^\infty(\Omega))}^2
\Big) \\
\nonumber
& + h^{2s_{\psi} + 2} \Big(\SemiNorm{\dptt \psi(\cdot, 0)}{H^{s_\psi + 1}(\Omega)}^2 + \sup_{t \in (0, T)} \SemiNorm{\dptt \psi}{H^{s_{\psi} + 1}(\Omega)}^2 + \Norm{\dpttt \psi}{L^2(0, T; H^{s_\psi + 1}(\Omega))}^2 \\
\nonumber 
& \qquad \qquad + \Norm{\dptt \psi}{L^{\infty}(0, T; H^{s_{\psi} + 1}(\Omega))}^2 + \Norm{\dpttt \psi}{L^2(0, T; H^{s_{\psi} + 1}(\Omega))}^2 \\
\nonumber 
& \qquad \qquad + \Norm{\psi}{H^2(0, T; H^{s_{\psi} + 1}(\Omega))}^2 \Norm{\dptt\psi}{L^\infty(0, T; L^\infty(\Omega))}^2
\\
& \qquad \qquad + \|\psi\|_{H^{3}(0, T; H^{s_{\psi} + 1}(\Omega))}^2 \Norm{\dpttt \psi }{L^2(0, T; L^\infty(\Omega))}^2
\Big),
\end{align*}
\end{subequations}
where the hidden constants are independent of~$h$ and~$\delta$.
Using the Sobolev embeddings~$H^2(0,T) \hookrightarrow C^1([0,T])$ and $H^3(0,T) \hookrightarrow C^2([0,T])$, and the fact that $h\in(0,\overline{h})$, we get the desired result.
\end{proof}

\section{Analysis of the semidiscrete HDG formulation for the Westervelt equation}\label{SEC::NONLINEAR_WESTERVELT}
We are now in a position to analyze the 
nonlinear semidiscrete formulation~\eqref{EQN::SEMI-DISCRETE-OPERATORS}.
The main idea consists of employing a Banach fixed-point argument applied to the mapping 
\[\mathcal{F}: \BK \ni (\psih^*, \bvh^*) \mapsto (\psih, \bvh), \]
$(\psih, \bvh)$ being the two first components (i.e., we omit the~$\lambdah$ component, which is uniquely determined by~$(\psih, \bvh)$; see also Remark~\ref{remark:well_definedness_lambda} bellow) 
of the unique solution to linear problem~\eqref{EQN::SEMI-DISCRETE-LINEARIZED} with discrete initial conditions as in Section~\ref{SEC::INITIAL_DATA_CHOICE}, $\bupsilon = 0$, $\varphi = 0$, and
\begin{equation*}
\alpha_h=  \pdt{1}\psi_{h}^*
\end{equation*}
from
\begin{equation}
\label{DEF::ball_West}
\begin{split}
\BK & := \left\{\vphantom{\int_0^t}\right.(\psi_h^*, \bv^*_h) \in W^{2,\infty}(0,T; \Sp) \times W^{1,\infty}(0,T; \Qp ):  (\psi_h^*, \dpt \psi^*_{h})_{|_{t=0}} = (\psi_{h}^0, \psih^{1}), \\ &
\sup_{t \in (0, T)}  \Big( \Norm{\bv-\bvh^*}{L^2(\Omega)^d}^2 + \Norm{\dpt \psi -{\dpt\psi}^*_h}{L^2(\Omega)}^2\Big)  \\
& \qquad \qquad \qquad \qquad  \leq 
C_0 \Big(h^{2s_{\psi}+2}\|\psi\|^2_{H^{2}(0, T; H^{s_{\psi} + 1}(\Omega))} 
+ h^{2s_{\bv}+2}\Norm{\bv}{H^{2}(0, T; H^{s_{\bv} + 1}(\Omega)^d)}^2\Big),
\\
&
\sup_{t \in (0, T)}  \Big( \Norm{\dpt\bv-\dpt\bvh^*}{L^2(\Omega)^d}^2 + \Norm{\pdt2 \psi -{\pdt2\psi}^*_h}{L^2(\Omega)}^2\Big) \\
& \qquad \qquad  \qquad \qquad \leq C_1 \Big(h^{2s_{\psi}+2}\|\psi\|^2_{H^{3}(0, T; H^{s_{\psi} + 1}(\Omega))} + h^{2s_{\bv}+2}\Norm{\bv}{H^{3}(0, T; H^{s_{\bv} + 1}(\Omega)^d)}^2\Big)	\left.\vphantom{\int_0^t}\right\},
\end{split}
\end{equation}
which is a ball centered at the exact solution~$(\psi,\bv) \in H^{3}(0, T; H_0^1(\Omega)\cap H^{s_{\psi} + 1}(\Omega)) \times H^{3}(0, T; H^{s_{\bv} + 1}(\Omega)^d)$
for some~$s_{\psi}$, $s_{\bv} \in (\frac{d}2-1, p]$.  
In the definition of~$\BK$, $C_0$ and~$C_1$ are positive constants independent of~$h$ and~$\delta$ that will be fixed in the proof of Theorem~\ref{THM::FIXED-POINT}.

Next theorem concerns the existence and uniqueness of the solution to the semidiscrete formulation~\eqref{EQN::SEMI-DISCRETE-OPERATORS}. 
Moreover, it provides optimal \emph{a priori} error estimates due to the definition of the ball~$\BK$. We denote by~$I_h$ the Lagrange interpolation operator in~$\Sp$. 
In particular, we {will} use the approximation result in~\cite[{Thm.~4.4.20}]{Brenner_Scott_2008} and the inverse estimate in~\cite[Thm.~4.5.11]{Brenner_Scott_2008}.

\begin{theorem}\label{THM::FIXED-POINT}
Let~$\delta \in [0, \bar{\delta})$, $p > \frac{d}2-1$, and~$s_{\psi}$, $s_{\bv} \in (\frac{d}{2}-1, p]$. 
Assume {that}~$(\psi,\bv) \in H^{{3}}(0, T; H_0^1(\Omega)\cap H^{s_{\psi} + 1}(\Omega)) \times H^{3}(0, T; H^{s_{\bv} + 1}(\Omega)^d)$ is the solution to the Westervelt equation in~\eqref{EQN::MODEL-PROBLEM-MIXED} for suitable initial conditions~$(\psi,\psi_t)_{|_{t=0}} =(\psi_0,\psi_1)$. 
Furthermore, let the discrete initial conditions $(\psih,\dpt \psih)_{|_{t = 0}}$ 
be chosen as in Section~\ref{SEC::INITIAL_DATA_CHOICE}. 
Then, there exist $T>0$, 
\[\overline{h} = \overline{h}\big(\|\psi\|_{H^{3}(0, T; H^{s_{\psi} + 1}(\Omega))} , \Norm{\bv}{H^{{3}}(0, T; H^{s_{\bv} + 1}(\Omega)^d)}\big) < 1, \quad \text{and} \quad 0 < M = M({k}, T),\]   
such that, for~$0<h<\overline{h}$ and 
\begin{equation*}		\int_0^{T}\|\pdt3\psi(s)\|^2_{L^\infty(\Omega)} \ds + \sup_{t \in  (0,T)}\|\pdt2\psi(t)\|^2_{L^\infty(\Omega)} + \int_0^{T}\|\pdt2\psi(s)\|^2_{L^\infty(\Omega)} \ds + \sup_{t \in  (0,T)}\|\pdt1\psi(t)\|^2_{L^\infty(\Omega)}\leq M,
\end{equation*}
there is a unique {solution}~$(\psi_h, \bv_h,\lambdah)\in \BK \times 
W^{1,\infty}(0,T;\Mp)$ 
to the semidiscrete HDG formulation~\eqref{EQN::SEMI-DISCRETE-OPERATORS} for 
some constants~$C_0,\, C_1 > 0$ in the definition of~$\BK$ that are independent 
of~$h$ and~$\delta$.

\end{theorem}
\begin{proof}
We proceed by using a Banach fixed-point argument.
The ball~$\BK$ is nonempty as it contains the HDG projection of the exact solution thanks to the estimates given in Lemma~\ref{LEMMA::Projections_Errors}. 

We split the proof into three parts. The first two are intended to prove the 
existence and uniqueness of a fixed point. The third part discusses the 
reconstruction of $\lambdah$.

\paragraph{Part I: Self-mapping.}
Let $(\psi_h^*, \bv^*_h) \in \BK$ and set $$(\psi_h, \bv_h) = \mathcal{F} (\psi_h^*, \bv^*_h).$$
To show the self-mapping property, we 
use the error estimates in Theorem~\ref{THM::ERROR-ESTIMATES}. 
We first verify that its assumptions hold.
We start by considering the nondegeneracy assumption in~\eqref{EQN::NON-DEGENERACY-ALPHA}. Using the triangle inequality, the quasi-uniformity of the mesh, and the stability and inverse estimates in~\cite[Thm.~4.4.20, Thm.~4.5.11]{Brenner_Scott_2008} for the Lagrange interpolation operator, we obtain
\begin{equation}\label{EQN::FIXED_POINT_NONDEG}
	\begin{split}
	\Norm{\alpha_h}{L^\infty(0,T;L^\infty(\Omega))} \leq &\, \|\pdt1\psi_{h}^*-I_h \pdt1\psi\|_{L^\infty(0,T;L^\infty(\Omega))}+ \|I_h \pdt1\psi\|_{L^\infty(0,T;L^\infty(\Omega))}
	\\
	\lesssim &\, h^{-d/2} \|\pdt1\psi_{h}^*-I_h \pdt1\psi\|_{L^\infty(0,T;L^2(\Omega))}+\|I_h \pdt1\psi\|_{L^\infty(0,T;L^\infty(\Omega))}\\
	\lesssim &\, h^{-d/2} \|\pdt1\psih^*-\pdt1\psi\|_{L^\infty(0,T;L^2(\Omega))}+ h^{-d/2} \|\pdt1\psi-I_h \pdt1\psi\|_{L^\infty(0,T;L^2(\Omega))} \\&+\|I_h \pdt1\psi\|_{L^\infty(0,T;L^\infty(\Omega))}.
	\end{split}
\end{equation}
Thus, we can guarantee that {the nondegeneracy condition in~\eqref{EQN::NON-DEGENERACY-ALPHA} holds with}
 \begin{equation} \label{EQN::Critical_case_p}
 \ulal = \olal = \overline{C} \left(\overline{h}^{s_{\psi}+1 - d/2}\|\psi\|_{H^{{3}}(0, T; H^{s_{\psi} + 1}(\Omega))} + \overline{h}^{s_{\bv}+1-d/2}\Norm{\bv}{H^{{2}}(0, T; H^{s_{\bv} + 1}(\Omega)^d)} + M^{1/2} \right) \in \Big(0, \frac{1}{2|k|}\Big),
 \end{equation}
for sufficiently small~$M$ and~$\overline h$, 
and some positive constant~$\overline{C}$ depending on~$C_0$ and~$C_1$, but not on~$h$ or~$\delta$.

 Similarly, {the smallness} assumptions {in}~\eqref{EQN::ENERGY-ASSUMPTION} and~\eqref{EQN::ERROR-HIGH-ASSUMPTION} can be shown to hold provided~$M$, $\overline{h}$, and {the} final time~$T$ are sufficiently small. 
 Assumption~\ref{Assumption_approx} is naturally verified since~$(\psi_h^*, \bv^*_h) \in \BK$.
 Therefore, Theorem~\ref{THM::ERROR-ESTIMATES} ensures the self-mapping property of~$\mathcal{F}$ (i.e., $\mathcal{F}(\BK) \subseteq \BK$) provided that~$C_0$ and~$C_1$ are large enough, and~$M$ is sufficiently small.

\paragraph{Part II: Strict contractivity.}
 Contractivity of the mapping~$\mathcal{F}$ follows similarly as in~\cite[Thm.~5.1]{Meliani_Nikolic_2023}, where the~$\delta$-robustness of the mixed FEM 
 for the Westervelt equation was proven. 
Indeed, one can obtain the contractivity of~$\mathcal{F}$ with respect to the lower topology $\sup_{t \in (0,T)}\calE_h^{(0)}[\cdot, \cdot, \cdot](t)$ by reducing~$M$ and~$\overline{h}$. 
The arguments showing the closedness of~$\BK$ with respect to the lower topology are analogous to \cite[Thm.~1.4]{Kaltenbacher_Nikolic_2022}. This shows that the fixed-point problem has a unique solution in~$\BK$, which solves the nonlinear problem~\eqref{EQN::SEMI-DISCRETE-OPERATORS}

\paragraph{Part III: Reconstructing $\lambdah$.}
Parts I and II ensure the existence of a unique fixed point $(\psih,\bvh) \in 
\BK$ to the mapping $\mathcal{F}$. To finish constructing the solution to the 
semidiscrete HDG formulation~\eqref{EQN::SEMI-DISCRETE-OPERATORS} we reconstruct 
$\lambdah$ as a function of $(\psih,\bvh)$ uniquely through 
\eqref{EQN::SEMI-DISCRETE-OPERATORS-3}. The triplet $(\psih, 
\bvh,\lambdah)\in\BK \times W^{1,\infty}(0,T;\Mp)$ thus constructed is the 
unique solution to \eqref{EQN::SEMI-DISCRETE-OPERATORS}.
\end{proof}

Along the lines of the analysis performed in~\cite[{\S4}]{Nikolic_2023} for the conforming FEM, we state here a corollary of the previous existence and uniqueness 
theorem, which will be useful in Section~\ref{SEC::ASYMPTOTIC_LIMITS} below for establishing the rate of convergence as~$\delta \to 0^+$.

\begin{corollary}\label{CORR:BOUNDS_NONLINEAR_SOLUTION}
    Under the assumptions of Theorem~\ref{THM::FIXED-POINT}, the solution~$(\psih,\bvh,\lambdah)$ to~\eqref{EQN::SEMI-DISCRETE-OPERATORS} satisfies
    \begin{equation}
    \label{EQN::BOUND-DPTT-PSIH}
\|\psihtt\|_{L^\infty(0,T;L^\infty(\Omega))} \leq C (\|\psi\|_{H^{3}(0, T; H^{s_{\psi} + 1}(\Omega))} + \Norm{\bv}{H^{3}(0, T; H^{s_{\bv} + 1}(\Omega)^d)}),
\end{equation}
    where $C>0$ does not depend on $h$ or $\delta$. Furthermore, the following bound holds:
    \begin{equation}
    \label{EQN::BOUND-DPT-PSIH}
    \|\psiht\|_{L^\infty(0,T; L^\infty(\Omega))} \leq \olal.
    \end{equation}
\end{corollary}
\begin{proof}
    The uniform-in-$h$-and-$\delta$ bound{s} follow from the use of inverse estimates as in~\eqref{EQN::FIXED_POINT_NONDEG}.
\end{proof}

We end this section showing that the solution to the semidiscrete 
formulation~\eqref{EQN::SEMI-DISCRETE-OPERATORS} from 
Theorem~\ref{THM::FIXED-POINT} is energy stable.
In the proof of the next result, we use the embedding~$H^3(0, T; H_0^1(\Omega) 
\cap H^2(\Omega)) \hookrightarrow C^2([0, T]; H_0^1(\Omega) \cap H^2(\Omega))$.

\begin{lemma}[Energy stability]\label{LEMMA::ENERGY-STABILITY}
Let the assumptions of Theorem~\ref{THM::FIXED-POINT} hold. Moreover, assume 
that the solution~$(\psi, \bv)$ to the Westervelt equation 
in~\eqref{EQN::MODEL-PROBLEM-MIXED} belongs to~$H^3(0, T; H_0^1(\Omega) \cap 
H^2(\Omega)) \times H^3(0, T; H^2(\Omega)^d)$.
Then, there exists a constant~$C_S > 0 $ independent of~$h \in (0,\overline h)$ and $\delta\in [0,\overline{\delta})$ such that
\begin{subequations}
\begin{align}
\label{EQN::LOW-ORDER-ENERGY-NONLINEAR}
\sup_{t \in (0, T)}\calE_h^{(0)}[\psih, \bvh, \lambdah](t) 
& \leq C_S (\|\psi_0\|_{H^2(\Omega)}+ \|\psi_1\|_{H^2(\Omega)}), \\
\label{EQN::HIGH-ORDER-ENERGY-NONLINEAR}
\sup_{t\in (0, T)} \calE_h^{(1)}[\psih, \bvh, \lambdah](t)
& \leq C_S (\|\psi_1\|_{H^2(\Omega)}+ \|\psi_{tt}(\cdot,0)\|_{H^2(\Omega)}),
\end{align}
\end{subequations}
{with~$\alphah = \dpt \psih$ in the definition of~$\calE_h^{(0)}[\psih, \bvh, \lambdah](t)$ and~$\calE_h^{(1)}[\psih, \bvh, \lambdah](t)$.
}
\end{lemma}
\begin{proof}
The proof follows by considering the solution to the nonlinear semidiscrete problem in~\eqref{EQN::SEMI-DISCRETE-OPERATORS} as the solution to the linearized problem in~\eqref{EQN::SEMI-DISCRETE-LINEARIZED} with~$\alpha_h=  \pdt{1}\psih.$ 
We can then proceed similarly as in Section~\ref{SUBSECT::WELL-POSEDNESS} to 
deduce that~$(\psih,\bvh,\lambdah) \in W^{3,1}(0,T; \Sp)\times  W^{3,1}(0,T; \Qp)\times W^{3, 1}(0, T; \Mp)$. 
By using similar arguments to those for the low- and high-order energy stability estimates in Theorem~\ref{THM::ENERGY-STABILITY}, we get
\begin{subequations}
\begin{align}
\label{EQN::AUX-BOUND-STAB-NONLINEAR-1}
\sup_{t\in (0, T)} \calE_h^{(0)}[\psih, \bvh, \lambdah](t) \leq (1 - \sigma_0)^{-1} \calE_h^{(0)}[\psih, \bvh, \lambdah](0), \\
\label{EQN::AUX-BOUND-STAB-NONLINEAR-2}
\sup_{t\in (0, T)} \calE_h^{(1)}[\psih, \bvh, \lambdah](t) \leq (1 - \sigma_0)^{-1} \calE_h^{(1)}[\psih, \bvh, \lambdah](0).
\end{align}
\end{subequations}
Therefore, it only remains to bound the initial discrete energies.

The following estimate follows from the stability of the discrete HDG elliptic problem in~\eqref{EQN::DISCRETE-INITIAL-CONDITION}:
\begin{equation}\label{EQN::INITIAL_ENERGY_STAB_VECTOR}
\begin{aligned}
\Norm{\bvhi}{L^2(\Omega)^d}^2 + \Norm{\tau^{\frac12} (\lambdahi - \psihi)}{L^2((\partial \Th)^{\mathcal{I}})}^2 & + \Norm{\tau^{\frac12}\psihi}{L^2((\partial \Th)^{\calD})}^2 \\
& \leq \frac12 \Norm{\Delta \psi_i}{L^2(\Th)}^2 + \frac12 \Norm{\psihi}{L^2(\Omega)}^2 \quad \text{ for }i = {0, 1}.
\end{aligned}
\end{equation}

Using the triangle inequality and the error estimate in~\cite[Cor.~2.7]{Cockburn_Dong_Guzman_2008} for second-order elliptic problems, we obtain
\begin{equation}
\label{EQN::BOUND-PSIHi}
\Norm{\psihi}{L^2(\Omega)} \leq \Norm{\psihi - \psi_i}{L^2(\Omega)} + \Norm{\psi_i}{L^2(\Omega)} \leq \max\{1,C h^{2}\} \Norm{\psi_i}{H^2(\Omega)},
\end{equation}
for~$i={0,1}$. This shows that we can estimate the right-hand side of \eqref{EQN::INITIAL_ENERGY_STAB_VECTOR} independently of $h$.
In particular, we can estimate
\begin{equation}
\label{EQN::BOUND-DPT-PSIH-0}
\begin{split}
\Norm{\sqrt{1 + 2 k {\dpt \psih}(\cdot, 0)}\psih^{(1)}}{L^2(\Omega)}^2 & \leq (1 
+ 2|k| \olal) (\Norm{\psi_1}{L^2(\Omega)} + \Norm{\psih^{(1)} - 
\psi_1}{L^2(\Omega)} ) \\
& \leq {(1 + 2|k| \olal)\max\{1, C h^{2}\}} \Norm{\psi_1}{H^2(\Th)},
\end{split}
\end{equation}
for some  
positive constant~$C$ independent of~$h$. Bound~\eqref{EQN::LOW-ORDER-ENERGY-NONLINEAR} then follows by combining~\eqref{EQN::AUX-BOUND-STAB-NONLINEAR-1}, bounds~\eqref{EQN::INITIAL_ENERGY_STAB_VECTOR} 
and~\eqref{EQN::BOUND-PSIHi} for~$i = 0$, and~\eqref{EQN::BOUND-DPT-PSIH-0}.

By the triangle inequality, we get
\begin{equation}
\label{EQN::PSITTH_BOUND}
\begin{aligned}
\Norm{\sqrt{1 + 2k \dpt \psih(\cdot, 0)}\dptt \psih(\cdot, 0)}{L^2(\Omega)} 
\leq & \  \Norm{\sqrt{1 + 2k \dpt \psih(\cdot, 0)}\dptt \psi(\cdot, 0)}{L^2(\Omega)} \\
& + \Norm{\sqrt{1 + 2k \dpt \psih(\cdot, 0)}\dptt \dpsi (\cdot, 0)}{L^2(\Omega)}\\
& + \Norm{\sqrt{1 + 2k \dpt \psih(\cdot, 0)}\dptt \etapsih(\cdot, 0)}{L^2(\Omega)}.
\end{aligned}
\end{equation}
The third term on the right-hand side of the above inequality satisfies
\begin{equation*}
 \Norm{\sqrt{1 + 2k {\dpt \psih}(\cdot, 0)}\dptt \etapsih(\cdot, 0)}{L^2(\Omega)}^2 \leq \calE_h^{(1)}[\etapsih, \etavh, \etalambdah](0),
\end{equation*}
which can be bounded using the approximation properties in Lemma~\ref{LEMMA::Projections_Errors}  of the HDG projection~$\piHDG$
due to~\eqref{EQN::INITIAL-HIGH-ORDER-BOUND}. Moreover, the following estimates hold:
\begin{subequations}
\begin{align*}
\Norm{\dpt \dbv (\cdot,0)}{L^2(\Omega)} & \lesssim h\Norm{\psi_1}{H^{2}(\Omega)}, 
\\
\Norm{\pdt{2}\dpsi(\cdot,0)}{L^2(\Omega)} & \lesssim {h} \big(\SemiNorm{\pdt2\psi(\cdot,0)}{H^1(\Omega)} + \Norm{\pdt2\psi(\cdot,0)}{H^{2}(\Omega)}\big).
\end{align*}
\end{subequations}

Introducing these bounds into \eqref{EQN::PSITTH_BOUND}, combining it with bounds~\eqref{EQN::INITIAL_ENERGY_STAB_VECTOR} and~\eqref{EQN::BOUND-PSIHi} for~$i = 1$, and using the nondegeneracy of~$\dpt \psih$ complete the proof of bound~\eqref{EQN::HIGH-ORDER-ENERGY-NONLINEAR}.
\end{proof}

\begin{remark}[Minimum degree of approximation]
\label{REM::MINIMUM-DEGREE}
The condition~$p > \frac d 2 -1$ in the statement of Theorem~\ref{THM::FIXED-POINT}, combined with the restriction~$d \in \{2,3\}$ on the spatial dimension, imposes that the degree of 
approximation must satisfy~$p\geq 1$.

Nevertheless, in the case~$d=2$, we can ensure the nondegeneracy of~$\dpt \psih$ even for~$p = \frac d 2 -1 = 0$, by assuming smallness of the exact solution~$ \|\psi\|_{H^{{3}}(0, T; H^{s_{\psi} + 1}(\Omega))} + \Norm{\bv}{H^{{2}}(0, T; H^{s_{\bv} + 1}(\Omega)^d)}$; see equation~\eqref{EQN::Critical_case_p}. This is relevant in practice, as is shown in the numerical experiments of Section~\ref{SEC::NUMERICAL_EXPERIMENTS}.
\eremk
\end{remark}

\begin{remark}[{Omission of~$\lambdah$}]\label{remark:well_definedness_lambda}
In the definition of the ball~$\BK$, we have omitted the component~$\lambdah$ of the solution to the linearized semidiscrete problem in~\eqref{EQN::SEMI-DISCRETE-LINEARIZED}, as Theorem~\ref{THM::ERROR-ESTIMATES} does not provide an error control 
for this component.
Nonetheless, 
given the fixed-point~$(\psih,\bvh)$  of the mapping~$\mathcal{F}$, 
which solves the nonlinear semidiscrete formulation 
in~\eqref{EQN::SEMI-DISCRETE-OPERATORS}, the component~$\lambdah$ is uniquely 
determined by~$(\psih, \bvh)$ through~\eqref{EQN::SEMI-DISCRETE-OPERATORS-3}, as 
was used in the proof of Theorem~\ref{THM::FIXED-POINT}. 
In fact, one can also define the mapping~$\mathcal{F}$ in terms of the first component~$\psih$ only, as the nonlinearity solely depends on such a component.
\eremk
\end{remark}

\section{Asymptotic behaviour at the vanishing viscosity limit}\label{SEC::ASYMPTOTIC_LIMITS}
This section is dedicated to the proof of convergence of the numerical scheme as~$\delta \to 0^+$. We denote in this section by $(\psih^{(\delta)}, \bvh^{(\delta)},\lambdah^{(\delta)})$ the solution to the semidiscrete formulation~\eqref{EQN::SEMI-DISCRETE-OPERATORS},
where we have stressed the dependence of the solution on the parameter~$\delta$. Then, {we denote} the difference $$(\opsih, \obvh, \olambdah) = (\psih^{(\delta)}- \psih^{(0)}, \bvh^{(\delta)} - \bvh^{(0)}, \lambdah^{(\delta)} - \lambdah^{(0)}),$$
which satisfies the following system of equations:
\begin{subequations}
\label{EQN::DIFFERENCE_EQ}
\begin{align}
\label{EQN::DIFFERENCE_EQ-1}
\bm_h(\obvh, \brh) + b_h(\opsih, \brh) + e_h(\olambdah, \brh) & = 0 & & \hspace{\fill} \forall \brh \in \Qp, \\
\nonumber
m_h((1 + 2k\dpt\psih^{(\delta)}) \pdt2\opsih, \wh) + m_h(2k\dpt\opsih \pdt2\psih^{(0)}, \wh)  & & \\ 
\label{EQN::DIFFERENCE_EQ-2}
- c^2b_h(\wh, \obvh) + c^2 s_h(\opsih, \wh)
   + c^2 f_h(\olambdah, \wh)  & = \delta F(\wh)  & & \forall \wh \in \Sp, 
\\
\label{EQN::DIFFERENCE_EQ-3}
- e_h(\muh, \obvh) + f_h(\muh, \opsih) + g_h(\olambdah, \muh) & = 0 & & \forall \muh \in \Mp,
\end{align}
\end{subequations}
with zero initial conditions. Above, the forcing
term~$F(\wh)$ is given by 
\[F(\wh) =   b_h(\wh, \bvht^{(\delta)}) - s_h(\psiht^{(\delta)}, \wh)
   - f_h(\lambdaht^{(\delta)}, \wh).\]

\begin{theorem}[$\delta$-convergence]
\label{THM::DELTA-CONVERGENCE}
    Let the assumptions of Theorem~\ref{THM::FIXED-POINT} hold, and let~$\overline{h}$ and~$T$ be fixed as in Theorem~\ref{THM::FIXED-POINT}. Then, the family of solutions 
   $\big\{(\psih^{(\delta)}, \bvh^{(\delta)},\lambdah^{(\delta)})\big\}_{\delta \in [0,\bar{\delta})}$ converges to $(\psih^{(0)}, \bvh^{(0)},\lambdah^{(0)})$ as~$\delta \rightarrow 0^+$,
   and 
   \begin{subequations}
   \label{EQN::DELTA-CONVERGENCE}
   \begin{align}
   \label{EQN::DELTA-CONVERGENCE-1}
       \sup_{t\in(0,T)}\calE_h^{(0)}(t) [\opsih, \obvh, \olambdah] & \leq C(T) \delta,\\
    \label{EQN::DELTA-CONVERGENCE-2}
       \sup_{t\in(0,T)} \|\opsih(t)\|^2_{L^2(\Omega)} & \leq C(T) \delta^2,
       \end{align}
   \end{subequations}
where~$C(T)$ is a generic constant that depends {exponentially} on~$T$.  
\end{theorem}
\begin{proof} 
{We prove each estimate separately.}

\paragraph{Proof of estimate~\eqref{EQN::DELTA-CONVERGENCE-1}.} The proof follows by a similar energy argument to that used to establish the low-order stability bound in {Appendix}~\ref{APP:ProofLowerOrder}. We differentiate~\eqref{EQN::DIFFERENCE_EQ-1} in time once and then take the test functions~$\brh = \obvh$, $\wh = \dpt\opsih$, and~$\muh = \dpt\olambdah$. Multiplying the first and third equations by~$c^2$, and summing the results, we get the identity
\begin{equation}
\label{EQN::CONV-IDENTITY}
\begin{split}
m_h((1 + & 2k\dpt\psih^{(\delta)}) \pdt2\opsih, \dpt\opsih)  + m_h(2k\dpt\opsih \pdt2\psih^{(0)}, \dpt\opsih )\\
& + c^2 \left(\bm_h(\dpt\obvh, \obvh) + s_h(\opsih, \dpt\opsih) + f_h(\olambdah, \dpt\opsih) + f_h(\dpt\olambdah, \opsih) + g_h(\dpt\lambdah, \dpt\olambdah)\right) \\
& =  \delta F(\dpt\opsih).
\end{split}
\end{equation}

We consider the following identities, which follow from the definition of the discrete bilinear forms in Section~\ref{SEC::DISCRETIZATION}:
\begin{align*}
m_h((1 + 2k\dpt\psih^{(\delta)}) \pdt2\opsih, \dpt\opsih)  = & \frac12 \frac{d}{dt} \Norm{\sqrt{1 + {2}k \dpt\psih^{(\delta)}} \dpt\opsih }{L^2(\Omega)}^2 - m_h(2k\pdt2 \psih^{(\delta)} \dpt\opsih, \dpt\opsih), \\
\nonumber 
\bm_h(\dpt\obvh, \obvh) &+ s_h(\opsih, \dpt\opsih) + f_h(\olambdah, \dpt\opsih) + f_h(\dpt\olambdah, \opsih) + g_h(\dpt\lambdah, \dpt\olambdah)\\
= & \frac12 \frac{d}{dt} \left(\Norm{\obvh}{L^2(\Omega)^d}^2 + \Norm{\tau^{\frac12} (\olambdah - \opsih)}{L^2((\partial \Th)^{\mathcal{I}})}^2 + \Norm{\tau^{\frac12}{\opsih}}{L^2((\partial \Th)^{\calD})}^2 \right).
\end{align*}
Using Corollary~\ref{CORR:BOUNDS_NONLINEAR_SOLUTION}, we can ensure that 
\[\int^t_0 m_h((1 + 2k\dpt\psih^{(\delta)}) \pdt2\opsih, \dpt\opsih) \ds \geq \frac{1-2|k|\olal}2 \|\dpt\opsih\|^2_{L^2(\Omega)} - 2 |k| \int^t_0\| \pdt2 \psih^{(\delta)} \|_{L^\infty(\Omega)}\|\dpt\opsih\|^2_{L^2(\Omega)}\ds, \]
where the negative term on the right-hand side will be controlled using the Gr\"onwall inequality.

It thus remains to control the term~$\delta F(\dpt \opsih)$. To this end, recall that the following equations hold:
\begin{subequations}
\begin{align*}
0 =& \, -\bm_h(\dpt \obvh, \brh) - b_h(\dpt \opsih, \brh) - e_h(\dpt \olambdah, \brh)  & & \forall \brh \in \Qp, \\
\nonumber
F(\wh) =&\,   b_h(\wh, \bvht^{(\delta)}) - s_h(\psiht^{(\delta)}, \wh)
   - f_h(\lambdaht^{(\delta)}, \wh)   & & \forall \wh \in \Sp, 
\\
0 =  &\, e_h(\muh, \dpt \bvh^{(\delta)}) - f_h(\muh, \dpt \psih^{(\delta)}) - g_h(\dpt\lambdah^{(\delta)}, \muh)  & &  \forall \muh \in \Mp. 
\end{align*}
\end{subequations}
Choosing $\brh = \dpt \bvh^{(\delta)}, \, \wh = \dpt \opsih$, and~$\muh = \dpt \olambdah$ above, and summing up the results yield the following identity:
\begin{equation*}\label{EQN::LIMITS_FORCING_TERM}
\begin{aligned}
F(\dpt\opsih) = -\bm_h(\dpt \obvh, \dpt \bvh^{(\delta)}) - s_h(\psiht^{(\delta)}, \dpt \opsih) -f_h(\lambdaht^{(\delta)}, \dpt \opsih) \\- f_h(\dpt \olambdah, \dpt \psih^{(\delta)}) - g_h(\dpt\lambdah^{(\delta)}, \dpt \olambdah).
\end{aligned}
\end{equation*}

By the definition of~$(\obvh, \opsih, \olambdah)$ and the Young inequality, we get
\begin{align*}
F(\dpt\opsih) = & - \Norm{\bvht^{(\delta)}}{L^2(\Omega)^d}^2 - \Norm{\tau^{\frac12}(\lambdaht^{(\delta)} - \psiht^{(\delta)})}{L^2((\partial \Th)^{\mathcal{I}})}^2 - \Norm{\tau^{\frac12}\psiht^{(\delta)}}{L^2((\partial \Th)^{\calD})}^2 \\
& + \big(\bvht^{(\delta)}, \bvht^{(0)}\big)_{\Omega} + \big(\tau (\psiht^{(\delta)} - \lambdaht^{(\delta)}), \psiht^{(0)} - \lambdaht^{(0)}\big)_{(\partial \Th)^{\mathcal{I}}} + \big(\tau \psiht^{(\delta)}, \psiht^{(0)}\big)_{(\partial \Th)^{\calD}} \\
\leq & \  3c^{-2} \calE_h^{(1)}[\psih^{(\delta)}, \bvh^{(\delta)}, \lambdah^{(\delta)}](t) + c^{-2} \calE_h^{(1)}[\psih^{(0)}, \bvh^{(0)}, \lambdah^{(0)}](t).
\end{align*}

Moreover, for all~$\tilde{t} \in (0, T)$,
\begin{equation*}
\int_0^{\tilde{t}} F(\dpt\opsih) \dt \leq 3c^{-2}\tilde{t}  \left(\sup_{t\in (0, \tilde{t})} \calE_h^{(1)}[\psih^{(\delta)}, \bvh^{(\delta)}, \lambdah^{(\delta)}](t) + 
\sup_{t\in (0, \tilde{t})} \calE_h^{(1)}[\psih^{(0)}, \bvh^{(0)}, \lambdah^{(0)}](t)\right).
\end{equation*}

Thus, by the energy stability estimates in Lemma~\ref{LEMMA::ENERGY-STABILITY}, the right-hand side is uniformly bounded with respect to both~$\delta$ and~$h$.
Inserting the above estimates into~\eqref{EQN::CONV-IDENTITY} and using the Gr\"onwall inequality yield estimate~\eqref{EQN::DELTA-CONVERGENCE-1}.

\paragraph{Proof of estimate~\eqref{EQN::DELTA-CONVERGENCE-2}.} 
We follow the approach in~\cite[Thm. 2 in \S 5.2]{Meliani_2023} for establishing asymptotic behavior of wave equations in weak topologies. 
For simplicity of notation, 
we introduce the operator
\[ \I{t'}u (t) := 
\left\{\begin{array}{ll}
			\int^{t'}_{t} u(s) \ds \quad  & \textrm{ if} \ \ 0\leq t\leq t',\\
			0 \quad  & \textrm{ if} \ \ t'\leq t\leq T.
		\end{array}\right. \]
We can then manipulate the system of equations in~\eqref{EQN::DIFFERENCE_EQ} to 
obtain
\begin{subequations}
\label{EQN::DIFFERENCE_EQ_rewritten}
\begin{align}
\bm_h(\I{t'}\obvh, \brh) + b_h(\I{t'}\opsih, \brh) + e_h(\I{t'}\olambdah, \brh)  & = 0 & & \forall \brh \in \Qp, \\
\nonumber
m_h(\pdt2\opsih + k \dpt(\dpt\opsih \dpt \psih^{(\delta)} + \dpt\opsih \dpt \psih^{(0)}), \wh)  & \\- c^2b_h(\wh, \obvh) + c^2 s_h(\opsih, \wh)
   + c^2 f_h(\olambdah, \wh)  & = \delta F(\wh)  & & \forall \wh \in \Sp, 
\\
- e_h(\muh, \obvh) + f_h(\muh, \opsih) + g_h(\olambdah, \muh) & = 0 & & \forall \muh \in \Mp,
\end{align}
\end{subequations}
where, in the second equation, we have used the identity
\[\dpt(\dpt\opsih \dpt \psih^{(\delta)} + \dpt\opsih \dpt \psih^{(0)}) = 2 \dpt\psih^{(\delta)}\pdt{2}\psih^{(\delta)} - 2 \dpt\psih^{(0)}\pdt{2}\psih^{(0)}. \]
We then choose the test functions~$\brh = \obvh$, $\wh = \I{t'}\opsih$, and~$\muh = \I{t'}\olambdah$.
Multiplying the first and third equations in~\eqref{EQN::DIFFERENCE_EQ_rewritten} by~$c^2$ and summing the results, we get the identity
\begin{align*}
       m_h(\pdt2\opsih + & k \dpt(\dpt\opsih \dpt \psih^{(\delta)} + \dpt\opsih \dpt \psih^{(0)}), \I{t'}\opsih)  \\ 
       & + c^2 \big(\bm_h(\I{t'}\obvh, \obvh) +s_h(\opsih, \I{t'}\opsih) + f_h(\olambdah, \I{t'}\opsih) + f_h(\I{t'}\olambdah, \opsih) + g_h(\olambdah, \I{t'}\olambdah)\big) \\
    & = \delta F(\I{t'}\opsih),
\end{align*}
which we integrate by parts in time on~$(0,t')$ to obtain
\begin{align}
\nonumber
      \int^{t'}_0 m_h  &([ 1 + k (\dpt \psih^{(\delta)} + \dpt \psih^{(0)})]\dpt\opsih, \opsih) \ds  \\ 
\label{EQN:ENERGY_LIMITS_LOWER_INTEGRATED}
      & + \int^{t'}_0 c^2 \big[\bm_h(\I{t'}\obvh, \obvh) +s_h(\opsih, \I{t'}\opsih) + {2}f_h(\olambdah, \I{t'}\opsih) 
      + g_h(\olambdah, \I{t'}\olambdah)\big] \ds = \delta F(\I{t'}\opsih).
    \end{align}
For the first term on the left-hand side, we make use of the following identity:
\begin{equation*}
\begin{aligned}\label{EQN:LIMITS::PSITT_TERM_LOWER}
m_h([1 + k (\dpt \psih^{(\delta)} + \dpt \psih^{(0)})]\dpt\opsih, \opsih)  = & \frac12 \frac{d}{dt} \Norm{\sqrt{1 + k (\dpt \psih^{(\delta)} + \dpt \psih^{(0)})} \opsih }{L^2(\Omega)}^2 \\ &- m_h(k (\pdt2 \psih^{(\delta)} + \pdt2 \psih^{(0)})\opsih, \opsih).\\
\end{aligned}
\end{equation*}

The positivity of~$1 + k (\dpt \psih^{(\delta)} + \dpt \psih^{(0)}) >0$ follows from bound~\eqref{EQN::BOUND-DPT-PSIH} in~Corollary~\ref{CORR:BOUNDS_NONLINEAR_SOLUTION}.
Further, {by} using the H\"older inequality {and bound~\eqref{EQN::BOUND-DPTT-PSIH}}, we obtain
\begin{equation*}
\begin{aligned}
m_h(k (\pdt2 \psih^{(\delta)} + \pdt2 \psih^{(0)})\opsih, \opsih) \leq &\|\pdt{2} \psih^{(\delta)} + \pdt2 \psih^{(0)}\|_{L^\infty(0,T; L^\infty(\Omega))} \|\opsih\|^2_{L^2(\Omega)} 
\\
\lesssim &\|\opsih\|^2_{L^2(\Omega)}.
\end{aligned}
\end{equation*}

Since~$\dpt \I{t'} u (t) = - u(t)$, we {can} write
\begin{equation*}\label{EQN:LIMITS::GRADIENT_TERM_LOWER}
\begin{aligned}
\bm_h(\I{t'}\obvh, \obvh) &+s_h(\opsih, \I{t'}\opsih) + {2}f_h(\olambdah, \I{t'}\opsih) 
+ g_h(\olambdah, \I{t'}\olambdah)\\= & - \frac12 \frac{d}{dt} \left(\Norm{\I{t'}\obvh}{L^2(\Omega)^d}^2 + \Norm{\tau^{\frac12} \I{t'}(\olambdah - \opsih)}{L^2((\partial \Th)^{\mathcal{I}})}^2 + \Norm{\tau^{\frac12}\I{t'}\opsih}{L^2((\partial \Th)^{\calD})}^2 \right).
\end{aligned}
\end{equation*}

It only remains to treat the forcing term~$F(\I{t'}\opsih)$. To this end, we proceed similarly as in the proof of estimate~\eqref{EQN::DELTA-CONVERGENCE-1}, and obtain
\begin{align*}
\nonumber
F(\I{t'} \opsih) = & -\bm_h(\I{t'} \obvh, \dpt \bvh^{(\delta)}) - s_h(\psiht^{(\delta)}, \I{t'} \opsih) \\
\nonumber
& - f_h(\lambdaht^{(\delta)}, \I{t'} \opsih) - f_h(\I{t'} \olambdah, \dpt \psih^{(\delta)}) - g_h(\dpt\lambdah^{(\delta)}, \I{t'} \olambdah)\\
\nonumber = & -\big(\I{t'} \obvh, \dpt \bvh^{(\delta)}\big)_{\Omega} -  \big(\tau \I{t'}(\olambdah - \opsih), \psiht^{(\delta)} - \lambdaht^{(\delta)}\big)_{(\partial \Th)^{\mathcal{I}}} - \big(\tau \I{t'}\opsih, \psiht^{(\delta)}\big)_{(\partial \Th)^{\calD}} \\ 
        \lesssim & \big(\Norm{\I{t'}\obvh}{L^2(\Omega)^d} + \Norm{\tau^{\frac12} \I{t'}(\olambdah - \opsih)}{L^2((\partial \Th)^o)} + \Norm{\tau^{\frac12}\I{t'}\opsih}{L^2((\partial \Th)^{\calD})} \big).
\end{align*}
In the last line, we have used the uniform-in-$\delta$ estimate of the high-order energy    ~$\calE_h^{(1)}[\psih^{(\delta)}, \bvh^{(\delta)}, \lambdah^{(\delta)}]$ {from Lemma~\ref{LEMMA::ENERGY-STABILITY}}. 
Putting the above estimates into {identity}~\eqref{EQN:ENERGY_LIMITS_LOWER_INTEGRATED}, we obtain
\begin{equation}\label{EQN::LIMITS_FORCING_TERM_LOWER}
\begin{split}
   & \Norm{\opsih (t') }{L^2(\Omega)}^2  + \Norm{\I{t'}\obvh(0)}{L^2(\Omega)^d}^2 + \Norm{\tau^{\frac12} \I{t'}(\olambdah - \opsih)(0)}{L^2((\partial \Th)^{\mathcal{I}})}^2 + \Norm{\tau^{\frac12}\I{t'}\opsih(0)}{L^2((\partial \Th)^{\calD})}^2  \\ 
   & \quad \lesssim \int^{t'}_0  \|\opsih\|^2_{L^2(\Omega)} \ds  + \delta\int^{t'}_0 \big(\Norm{\I{t'}\obvh}{L^2(\Omega)^d} + \Norm{\tau^{\frac12} \I{t'}(\olambdah - \opsih)}{L^2((\partial \Th)^{\mathcal{I}})} + \Norm{\tau^{\frac12}\I{t'}\opsih}{L^2((\partial \Th)^{\calD})} \big) \ds,
\end{split}
\end{equation}
where the hidden constant does not depend
on~$\delta$
or~$h$. 

In order to rewrite~\eqref{EQN::LIMITS_FORCING_TERM_LOWER} in a suitable form so as to be able to use the Gr\"onwall inequality, 
we introduce the time-reversed operator~$\widetilde{\operatorname{I}}_{t'}${,} which we define for an integrable function~$u$ and~$t\in(0,T)$ by 
\[\widetilde{\operatorname{I}}_{t'} u (t) := \I{t'} u (t'- t).\]

{By the definition of~$\widetilde{\operatorname{I}}_{t'}$, bound} \eqref{EQN::LIMITS_FORCING_TERM_LOWER} can be conveniently rewritten as
\begin{align*}
   & \Norm{\opsih (t') }{L^2(\Omega)}^2 + \Norm{\widetilde{\operatorname{I}}_{t'}\obvh(t')}{L^2(\Omega)^d}^2 + \Norm{\tau^{\frac12} \widetilde{\operatorname{I}}_{t'}(\olambdah - \opsih)(t')}{L^2((\partial \Th)^o)}^2 + \Norm{\tau^{\frac12}\widetilde{\operatorname{I}}_{t'}\opsih(t')}{L^2((\partial \Th)^{\calD})}^2 \\ 
   & \quad \lesssim \delta^2 +\int^{t'}_0 \big(  \|\opsih\|^2_{L^2(\Omega)} + \Norm{\widetilde{\operatorname{I}}_{t'}\obvh}{L^2(\Omega)^d}^2 + \Norm{\tau^{\frac12} \widetilde{\operatorname{I}}_{t'}(\olambdah - \opsih)}{L^2((\partial \Th)^{\mathcal{I}})}^2 + \Norm{\tau^{\frac12}\widetilde{\operatorname{I}}_{t'}\opsih}{L^2((\partial \Th)^{\calD})}^2 \big)\ds,
\end{align*}
where we {have} additionally used the Young inequality to get~$\delta^2$ on the right-hand side. The Gr\"onwall inequality yields then the desired result.
\end{proof}

\section{Numerical experiments}\label{SEC::NUMERICAL_EXPERIMENTS}
In this section, we assess the accuracy and robustness of the proposed method. 
In Section~\ref{SECT::FULLY-DISCRETE}, we present some details for the implementation of the fully discrete scheme obtained by combining the semidiscrete formulation~\eqref{EQN::SEMI-DISCRETE-OPERATORS} with the Newmark time-marching scheme. 
The~$h$- and~$\delta$-convergence of the proposed method are illustrated in Sections~\ref{SECT::NUM-h-convergence} and~\ref{SECT::DELTA-ERROR}, respectively.
In Section~\ref{SECT::WAVEFRONT}, we present an example of the effect of the nonlinearity parameter~$k$ on the solution.

Although our theory does not provide any superconvergence result, in the numerical experiments below, we consider the following local postprocessing technique (see~\cite[\S2.2]{Cockburn-ETAL-2018}): given the numerical approximation~$(\psih, \bvh, \lambdah)$ of the solution to~\eqref{EQN::MODEL-PROBLEM-MIXED} at some time~$t \geq 0$, we define~$\psih^* \in \mathcal{S}_h^{p+1}(\Th)$ such that, for all~$K \in \Th$, it satisfies
\begin{subequations}
\label{EQN::POSTPROCESSING}
\begin{align}
\int_K \nabla \psih^* \cdot \nabla q_{p+1} \dx & = \int_K \bvh \cdot \nabla q_{p+1}\dx \qquad  \forall q_{p+1} \in \Pp{p+1}{K}, \\
\int_K \psih^* \dx & = \int_K \psih \dx.
\end{align}
\end{subequations}

For the HDG discretization in~\cite{Cockburn-ETAL-2018} of the linear wave equation, the postprocessed variable~$\psih^*$ was shown to superconverge with order~$\mathcal{O}(h^{p+2})$ if~$p > 0$. 
Such a superconvergence is also numerically observed in~Section~\ref{SECT::NUM-h-convergence} below for the nonlinear Westervelt equation.

An object-oriented MATLAB implementation of the fully discrete scheme described in the next section was developed to carry out the numerical experiments in two-dimensional domains. 

\subsection{Fully discrete scheme\label{SECT::FULLY-DISCRETE}}
We use the predictor-corrector Newmark scheme in~\cite[\S5.4.2]{Kaltenbacher_2007} as time discretization. 
Let~$\Delta t$ be a fixed time step, $tol > 0$ be a given tolerance, $s_{\max}$ be a maximum number of linear iterations, and~$(\gamma, \beta)$ be the Newmark parameters with~$\gamma \in [0, 1]$ and~$\beta \in [0, 1/2]$. 
In the numerical experiments below, 
we use~$tol = 10^{-10}$ and~$s_{\max} = 100$, and it will be useful to consider an inhomogeneous forcing term~$\varphi : \QT \rightarrow \IR$.
For convenience, we use the dot notation for discrete approximations of time derivatives. 

In Algorithm~\ref{ALG::FULLY-DISCRETE}, we describe 
an implementation of the proposed fully discrete scheme.

\begin{algorithm}[H]
\caption{\sc Newmark-HDG fully discrete scheme}
\label{ALG::FULLY-DISCRETE}
\SetKw{INPUT}{Input:}
\SetKw{TO}{ to }
\SetKw{SET}{Set}
\SetKw{STOP}{stop}
\SetKw{COMPUTE}{Compute}
\SetKw{SOLVE}{Solve}
\SET a time step~$\Delta t > 0$. \\
\SET a tolerance~$tol > 0$ and a maximum number of linear iterations~$s_{\max} \in \IN$. \\
\COMPUTE the coefficient~$\mu = c^2(\Delta t)^2 \beta + \delta \gamma \Delta t$ and the number of time steps~$N_T = T/\Delta t$.\\
\COMPUTE the Schur complement matrices
\begin{equation*}
\calS_{\psi}  = S + B^T \bM^{-1} B, \quad \calA_{\lambda}  = G + E^T \bM^{-1} E, \quad \text{ and }  \quad \calR_{\lambda}  = F + B^T \bM^{-1} E.
\end{equation*} \\
\SOLVE the matrix systems\footnote{The matrix systems in~\eqref{EQN::MATRIX-SYSTEMS} can be solved completely in parallel due to the block-diagonal structure of the matrices~$M$, $\bM$, and~$S_{\psi}$.} for the pairs~$(\bX, Y)$ and~$(\overline{\bX}, \overline{Y})$ 
\begin{equation}
\label{EQN::MATRIX-SYSTEMS}
\begin{cases}
 (M + \mu S_{\psi}) Y = \mu \calR_{\lambda}, \\
 \bM \bX = E - BY.
\end{cases}
\qquad
\begin{cases}
\calS_{\psi} \overline{Y} = \calR_{\lambda},\\
\bM \overline{\bX} = E - B\overline{Y}.
\end{cases}
\end{equation} \\
\COMPUTE the auxiliary matrices
$$\calS_{\lambda, \mu} = G + 
E^T \bX -F^T Y  \quad \text{ and } \quad  \bar{\calS}_{\lambda} = G + E^T \overline{\bX} -F^T \overline{Y}.$$ \\
\COMPUTE the discrete initial conditions~$(\Psih^{(0)}, \Vh^{(0)}, \Lambdah^{(0)})$ and~$(\dot{\Psi}_h^{(0)}, \dot{\V}_h^{(0)}, \dot{\Lambda}_h^{(0)})$ by solving~\eqref{EQN::DISCRETE-INITIAL-CONDITION}. \\
\SOLVE the linear systems\footnote{Here, $N_h(\cdot)$ is the block-diagonal matrix described in Remark~\ref{REM::STRUCTURE-N}.} for~$(\ddot{\Psi}_h^{(0)}, \ddot{\Lambda}_h^{(0)})$
    \begin{equation*}
    N_h(\dot{\Psi}_h^{(0)}) \ddot{\Psi}_h^{(0)} = -c^2 \big(\calS_{\psi}\tPsih^{(0)} + \calR_{\lambda} \tLambdah^{(0)}\big)
    \quad \text{ and } \quad \calA_{\lambda} \ddot{\Lambda}_h^{(0)} = - \calR_{\lambda}^T \ddot{\Psi}_h^{(0)}.
    \end{equation*}
    \\
\For {$n = 0$ \TO $N_T$}{
\CommentSty{\textcolor{blue}{\%\%\% PREDICTOR STEP \%\%\%}}\\
\COMPUTE the approximations
\begin{alignat*}{3}
\widehat{\Psi}_h^{(n + 1)} & = \Psih^{(n)} + \Delta t \dot{\Psi}_h^{(n)} + \frac{(\Delta t)^2}{2} (1 - 2 \beta) \ddot{\Psi}_h^{(n)}, & \quad \widehat{\dot{\Psi}}_h^{(n + 1)} & = \dot{\Psi}^{(n)} + ( 1 - \gamma) \Delta t \ddot{\Psi}_h^{(n)},  \\
\widehat{\Lambda}_h^{(n + 1)} & = \Lambdah^{(n)} + \Delta t \dot{\Lambda}_h^{(n)} + \frac{(\Delta t)^2}{2} (1 - 2 \beta) \ddot{\Lambda}_h^{(n)},& \quad 
\widehat{\dot{\Lambda}}_h^{(n + 1)} & = \dot{\Lambda}^{(n)} + ( 1 - \gamma) \Delta t \ddot{\Lambda}_h^{(n)}, \\
\widetilde{\widehat{\Psi}}_h^{(n+1)} & = \widehat{\Psi}_h^{(n+1)} + \frac{\delta}{c^2} \widehat{\dot{\Psi}}_h^{(n + 1)}, & \quad \widetilde{\widehat{\Lambda}}_h^{(n+1)} & = \widehat{\Lambda}_h^{(n+1)} + \frac{\delta}{c^2} \widehat{\dot{\Lambda}}_h^{(n + 1)}.
\end{alignat*}\\ 
\COMPUTE the~$n$th step vector\footnote{Here, $\Phi^{n+1}$ is the vector representation of the forcing term~$\varphi$ at~$t = t_{n+1}$.}
$\calL^n = \Phi^{n+1}-c^2 \big(\calS_{\psi} \widetilde{\widehat{\Psi}}_h^{(n+1)} + \calR_{\lambda} \widetilde{\widehat{\Lambda}}_h^{(n+1)} \big).$ \\
\CommentSty{\textcolor{blue}{\%\%\% CORRECTOR STEP \%\%\%}}\\
\For {$s = 1$ \TO $s_{\max}$}{
\COMPUTE \hspace{0.68in} $R^{(n+1, s)} =  \big(M - N_h(\dot{\Psi}_h^{(n + 1, s)})\big) \ddot{\Psi}_h^{(n + 1, s)} + \calL^n$. \\
\SOLVE \footnote{The linear systems in lines~\textbf{16} and~\textbf{18} can be solved in parallel due to the block-diagonal structure of the matrix~$M + \mu S_\psi$.}
\hspace{0.2in}
 $(M + \mu \calS_{\psi}) Z^{(n + 1, s)} = R^{(n+1, s)}$. \\
\SOLVE\footnote{The linear system in line~\textbf{17} involves the solution of a statically condensed linear system, where the unknowns are associated with degrees of freedom on~$(d - 1)$-dimensional mesh facets only.} 
\hspace{0.5in}
$\calS_{\lambda, \mu} \ddot{\Lambda}_h^{(n+1, s + 1)} = -\calR_{\lambda}^T Z^{(n+ 1, s)}$. \\
\SOLVE \hspace{0.13in} $(M + \mu S_{\psi}) \ddot{\Psi}_h^{(n + 1, s + 1)} = R^{(n + 1, s)} - \mu \calR_{\lambda}\ddot{\Lambda}_h^{(n + 1, s + 1)}$. \\
\COMPUTE \hspace{0.55in} $\dot{\Psi}_h^{(n + 1, s + 1)} = \widehat{\dot{\Psi}}_h^{(n + 1)} + \gamma \Delta t \ddot{\Psi}_h^{(n + 1, s + 1)}$. \\
\CommentSty{\textcolor{blue}{\%\%\% STOPPING CRITERIA
\%\%\%}}\\
\If {$\Norm{\Psih^{n+1, s + 1} - \Psih^{n+1, s}}{} {\big/} \Norm{\Psih^{n+1, s + 1} }{} < tol $}
{\STOP}
}
}
\end{algorithm}
\subsection{\texorpdfstring{$h$}{h}-convergence \label{SECT::NUM-h-convergence}}
In order to assess the accuracy in space of the proposed method, 
we consider the Westervelt equation in~\eqref{EQN::MODEL-PROBLEM} on the domain~$Q_T = (0, 1)^2 \times (0, T)$, with parameters~$c = 100\text{ms}^{-1},$ $\delta = 6\times 10^{-9}\text{ms}^{-1}$, and~$k = 0.5 \text{s}^2\text{m}^{-2}$. We add a forcing term~$\varphi : Q_T \rightarrow \IR$ and set the initial data such that the exact solution is given by
\begin{equation}
\label{EQN::EXACT-SOLUTION}
\psi(x, y, t) = A \sin(\omega t) \sin(\ell x) \sin(\ell y),
\end{equation}
with~$A = 10^{-2} \text{m}^2 \text{s}^{-1}$, $\omega = 3.5\pi \text{Hz}$, and~$\ell = \pi \text{m}^{-1}$; cf. \cite[\S6]{Meliani_Nikolic_2023}.

We consider a set of structured simplicial meshes~$\{\Th\}_{h>0}$ for the spatial domain~$\Omega$, which we exemplify in Figure~\ref{fig:h-convergence}(left panel in the first row). We set the parameters~$(\gamma,\, \beta) = (1/2,\, 1/4)$ for the Newmark scheme, which guarantee second-order accuracy in time and unconditional stability in the linear setting (see, e.g., \cite[\S9.1.2]{Hughes_2000}). 
The time step is chosen as~$\Delta t = \mathcal{O}(h^{\frac{p+2}{2}})$, so as to balance the expected convergence rates of order~$\mathcal{O}(h^{p+2})$ for the postprocessed approximation~$\psih^*$ with the second-order accuracy of the Newmark scheme. 

In Figure~\ref{fig:h-convergence}, we show in~\emph{log-log} scale the following errors at the final time~$T = 1$s:
\begin{equation}
\label{EQN::h-ERRORS}
\Norm{\psi(\cdot, T) - \psih^{(N_T)}}{L^2(\Omega)}, \quad \Norm{\psi(\cdot, T) - {\psih^*}^{(N_T)}}{L^2(\Omega)}, \quad \Norm{\bv(\cdot, T) - \bvh^{(N_T)}}{L^2(\Omega)^2}.
\end{equation} 

For~$p = 0, 1, 2$, optimal convergence rates of order~$\mathcal{O}(h^{p+1})$ are obtained for the $L^2(\Omega)$-errors of~$\psih$ and~$\bvh$, which is in agreement with the \emph{a priori} error estimates derived in Section~\ref{SEC::NONLINEAR_WESTERVELT} for the semidiscrete HDG formulation. 
Moreover, when~$p > 0$, superconvergence of order~$\mathcal{O}(h^{p+2})$ is observed for the~$L^2(\Omega)$-error of the postprocessed variable~$\psih^*$ defined in~\eqref{EQN::POSTPROCESSING}.

\begin{figure}[!ht]
    \centering
    \includegraphics[width = 0.4\textwidth]{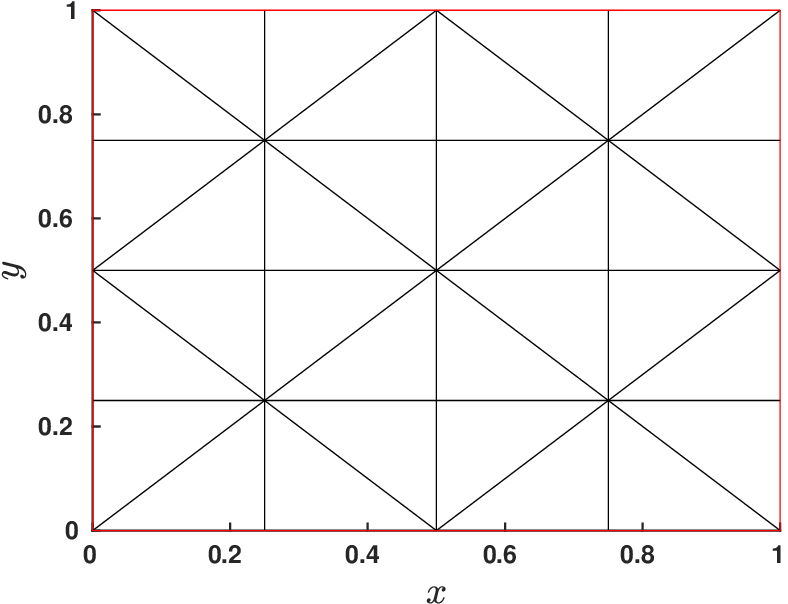}
    \hspace{0.2in}
    \includegraphics[width = 0.4\textwidth]{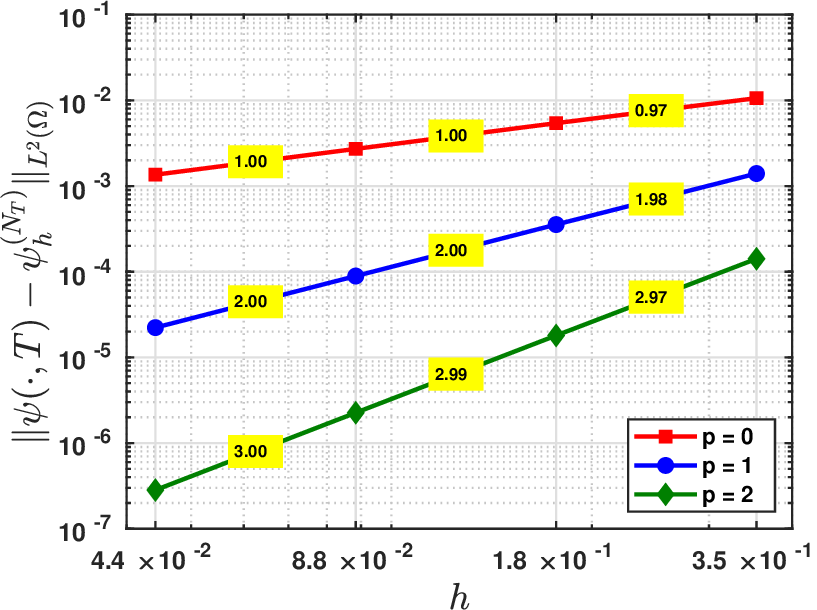}\\[2em]
    \includegraphics[width = 0.4\textwidth]{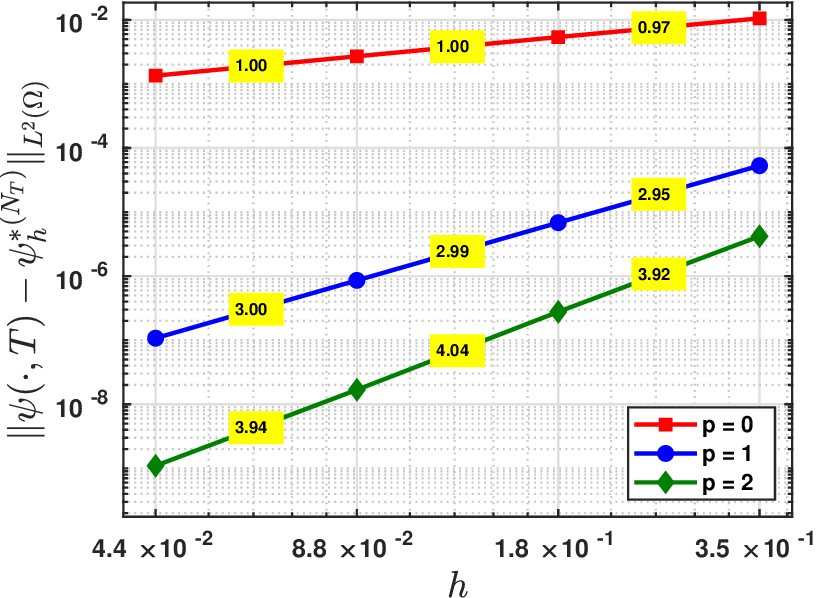}
    \hspace{0.2in}
    \includegraphics[width = 0.4\textwidth]{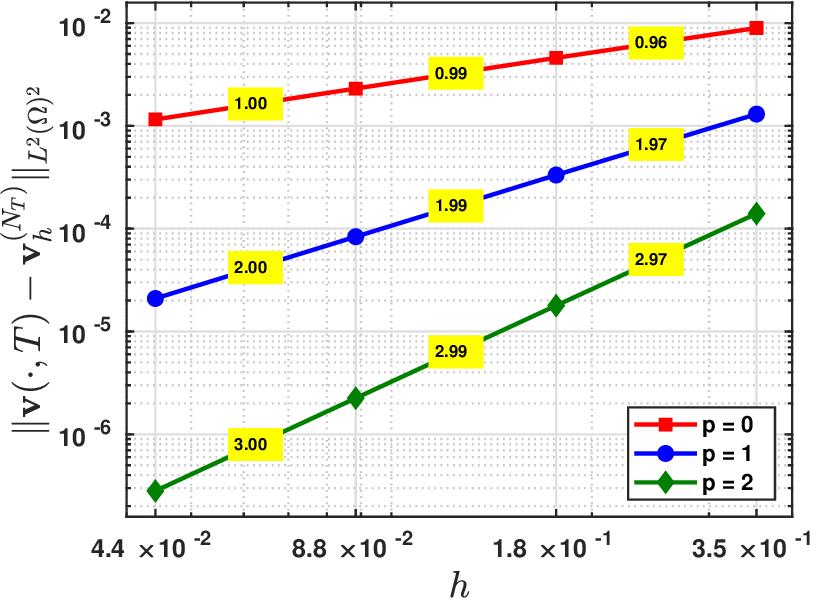}
    \caption{\textbf{First panel:} Example of the simplicial meshes used in the numerical examples. \textbf{Remaining panels:} $h$-convergence of the errors in~\eqref{EQN::h-ERRORS} at the final time~$T = 1$s for the test case with exact solution~\eqref{EQN::EXACT-SOLUTION}. The
numbers in the yellow rectangles denote the experimental rates of convergence.}
    \label{fig:h-convergence}
\end{figure}

\subsection{\texorpdfstring{$\delta$}{}-convergence \label{SECT::DELTA-ERROR}}
We now validate the convergence of the method when the sound diffusivity parameter~$\delta$ tends to zero.
To do so, we consider the Westervelt equation in~\eqref{EQN::MODEL-PROBLEM} on the domain~$\QT = (0, 1)^2 \times (0, T)$, with parameters~$c = 1 \text{m}\text{s}^{-1}$ and~$k = 0.3 \text{s}^2 \text{m}^{-2}$. The initial data are given by
\begin{equation}
\label{EQN::INITIAL-DELTA-CONVERGENCE}
\psi_0(x, y) = 10^{-2} \sin(\pi x) \sin(\pi y), \qquad \psi_1(x, y) = \sin(\pi x) \sin (\pi y),
\end{equation}
the spatial mesh is taken as in Figure~\ref{fig:h-convergence}(left panel in the first row), and the parameters~$(\gamma, \beta)$ and the time step is chosen as in the previous experiment; cf. \cite[\S2.4.2]{Dorich_Nikolic_2024}. 
We consider piecewise constant~$(p = 0)$ and piecewise linear~$(p = 1)$ approximations, and~$\delta = 10^{-2i}$ with~$i = 1, \ldots, 5$.

In Figure~\ref{fig:delta-convergence}, we show in \emph{log-log} scale the following errors computed at the final time~$T = 1$s:
\begin{equation}
\label{EQN::DELTA-ERRORS}
\Norm{\psih^{(\delta)} - \psih^{(0)}}{L^2(\Omega)} \quad \text{ and } \quad \Norm{\bvh^{(\delta)} - \bvh^{(0)}}{L^2(\Omega)^2}.
\end{equation}

Convergence rates of order~$\mathcal{O}(\delta)$ are observed for both errors. For~$p = 1$, these results are in agreement with estimate~\eqref{EQN::DELTA-CONVERGENCE-2}, and suggest that estimate~\eqref{EQN::DELTA-CONVERGENCE-1} may be not sharp. 
In fact, in~\cite[Thm.~2.2]{Dorich_Nikolic_2024}, convergence rates of order~$\mathcal{O}(\delta)$ were established for the standard finite element method by exploiting the relation~$\bvh = \nabla \psih$. 
Moreover, it is likely that the exact solution is more regular than assumed in Theorem~\ref{THM::FIXED-POINT}, in which case one could show full convergence rates of order~$\mathcal{O}(\delta)$ 
in~\eqref{EQN::DELTA-CONVERGENCE-1}, by deriving higher-order energy stability estimates. 
\begin{figure}[!ht]
    \centering
    \includegraphics[width = 0.4\textwidth]{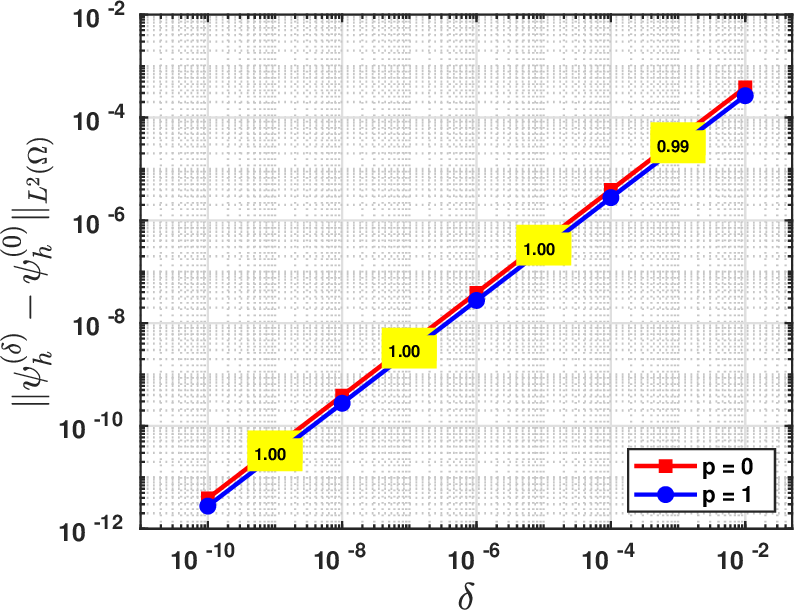} 
    \hspace{0.2in}
    \includegraphics[width = 0.4\textwidth]{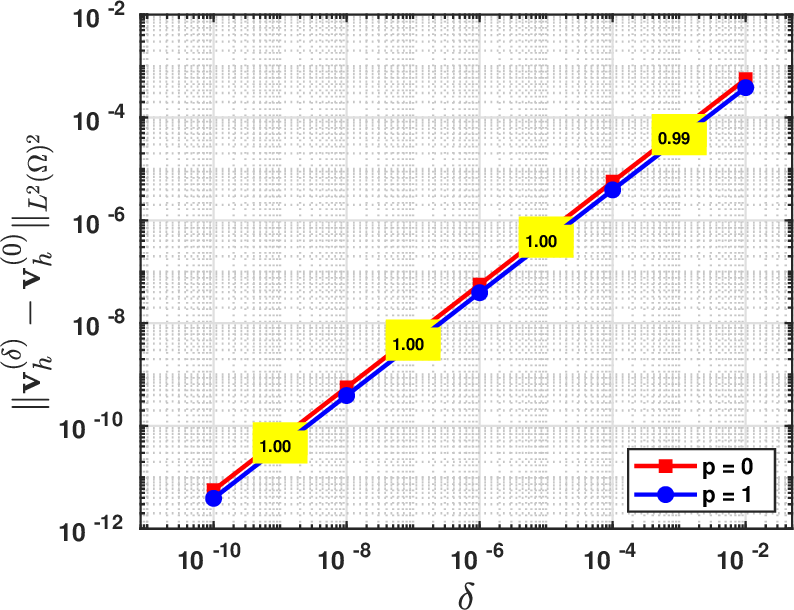} 
    \caption{$\delta$-convergence of the errors in~\eqref{EQN::DELTA-ERRORS} at the final time~$T = 1$s for the test case in Section~\ref{SECT::DELTA-ERROR}.}
    \label{fig:delta-convergence}
\end{figure}

\subsection{Steepening of a wavefront \label{SECT::WAVEFRONT}}
In this experiment, we illustrate the effect of the nonlinearity parameter~$k$ on the solution. We consider the Westervelt equation in~\eqref{EQN::MODEL-PROBLEM} on the domain~$\QT = (0, 1)^2 \times (0, T)$, with parameters~$c = 1500 \text{m}\text{s}^{-1}$, $\delta = 6\times 10^{-9} \text{m}\text{s}^{-1}$, and~$k = -10\text{s}^2 \text{m}^{-2}$. We consider homogeneous initial conditions and the following forcing term:
\begin{equation}
\label{EQN::SOURCE-TERM-STEEPENING}
\varphi(x, y, t) = \frac{a}{\sqrt{\sigma}} \exp(-\alpha t) \exp\Big(-\frac{(x - 0.5)^2 + (y - 0.5)^2}{2\sigma^2}\Big),
\end{equation}
where~$a = 400$, $\alpha = 5 \times 10^4$, and~$\sigma = 3 \times 10^{-2}$; cf. \cite[\S6]{Meliani_Nikolic_2023}.

We employ a simplicial mesh~$\Th$ with~$h \approx 8.83\times 10^{-2}$, a fixed time step~$\Delta t = 10^{-6}$, and~$p = 5$. 
In order to deal with the steepening of the wave, the parameters for the Newmark scheme are chosen as~$(\gamma, \beta) = (0.85, 0.45)$.
In Figure~\ref{fig:steepening-plot}(left panels), 
we show the approximation of~$\partial_t \psi$ obtained at~$t = 5 \times 10^{-5}$s and~$t = 2 \times 10^{-4}$s. In Figure~\ref{fig:steepening-plot}(right panels), we compare the approximation of~$\dpt{\psi}$ obtained for the nonlinear Westervelt equation ($k = -10$) and the damped linear wave equation ($k = 0$) along the line~$y = 0.5$. A steepening at the wavefront of the solution is clearly observed for the nonlinear model.

\begin{figure}[!ht]
    \centering
    \includegraphics[width = 0.45\textwidth]{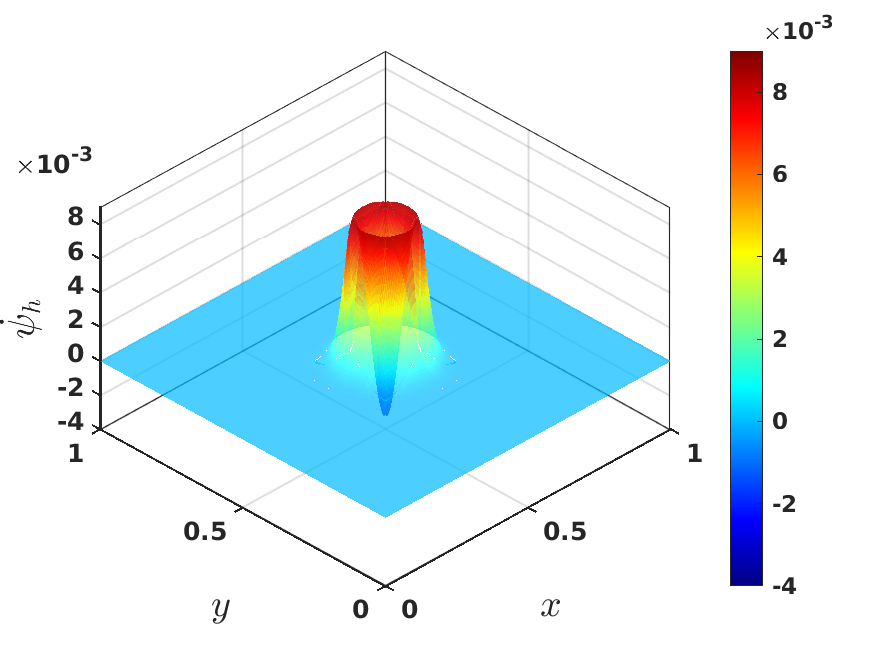}
    \hspace{0.1in}
    \includegraphics[width = 0.4\textwidth]{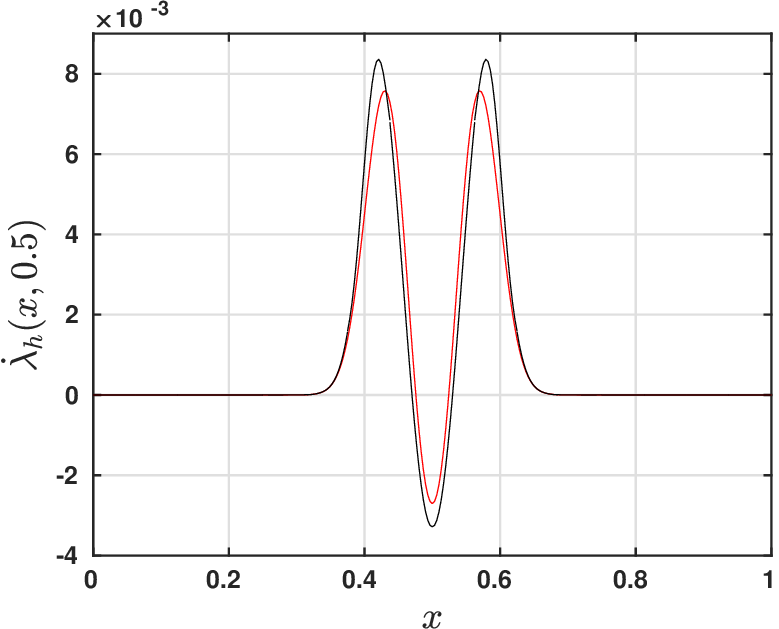}
    \\
    \includegraphics[width = 0.45\textwidth]{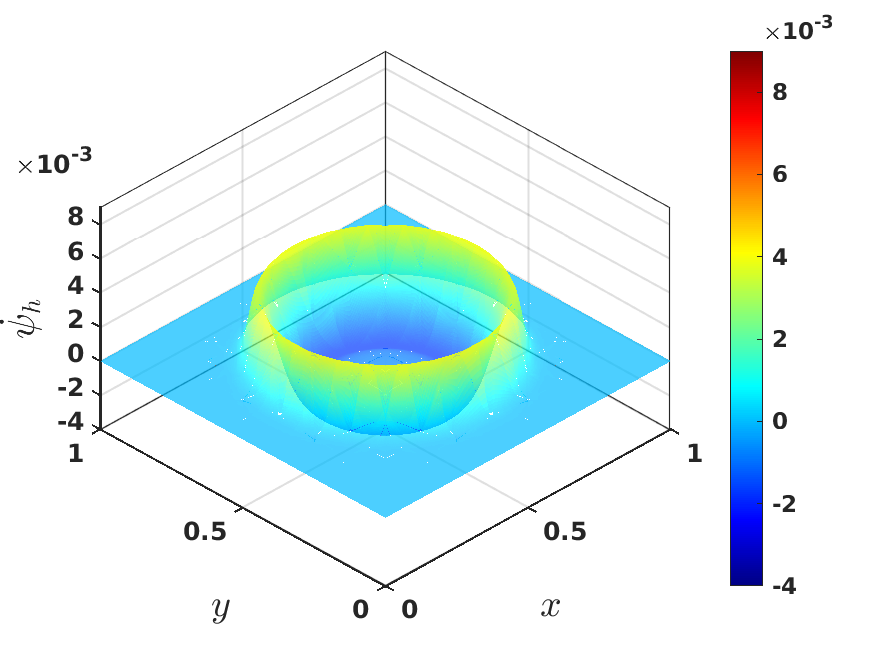}
    \hspace{0.2in}
    \includegraphics[width = 0.4\textwidth]{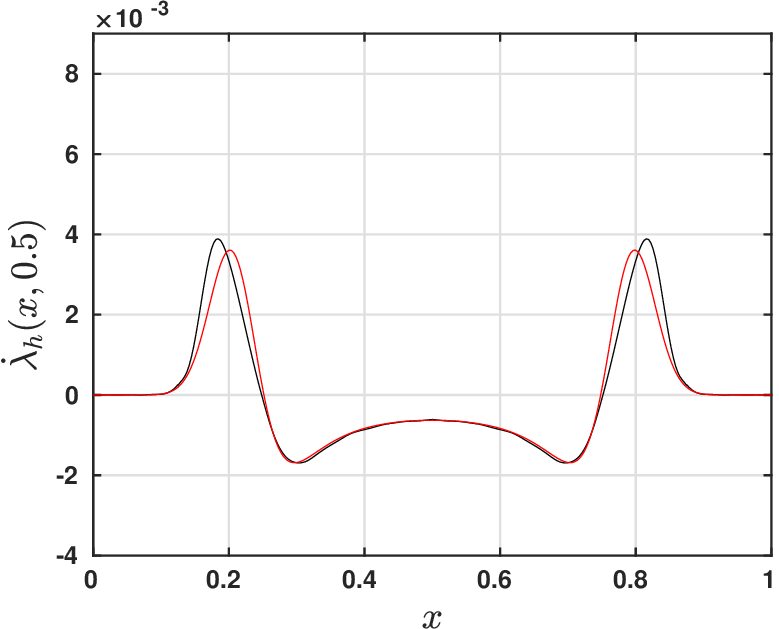}
    \caption{Results obtained at~$t = 5\times 10^{-5}$s (first row) and~$t = 2\times 10^{-4}$s (second row) for the test case in Section~\ref{SECT::WAVEFRONT}. \textbf{Left panels:} Approximation of~$\partial_t \psi$ obtained for~$p = 5$ and~$k = -10\text{s}^2\text{m}^{-2}$. \textbf{Right panels:} Comparison of the approximations obtained for the Westervelt equation (\textbf{black} lines) and the linear damped wave equation (\textbf{\textcolor{red}{red}} lines) along the line~$y = 0.5$.}
    \label{fig:steepening-plot}
\end{figure}

Since the forcing term~$\varphi$ in~\eqref{EQN::SOURCE-TERM-STEEPENING} is independent of~$\delta$, the~$\delta$-convergence estimates in Theorem~\ref{THM::DELTA-CONVERGENCE} are still valid. In Figure~\ref{fig:delta-convergence-wavefront}, we show the errors in~\eqref{EQN::DELTA-ERRORS} obtained at~$t = 10^{-4}$s. Convergence rates of order~$\mathcal{O}(\delta)$ are observed as in the numerical experiment of Section~\ref{SECT::DELTA-ERROR}.

\begin{figure}[!ht]
    \centering
    \includegraphics[width = 0.4\textwidth]{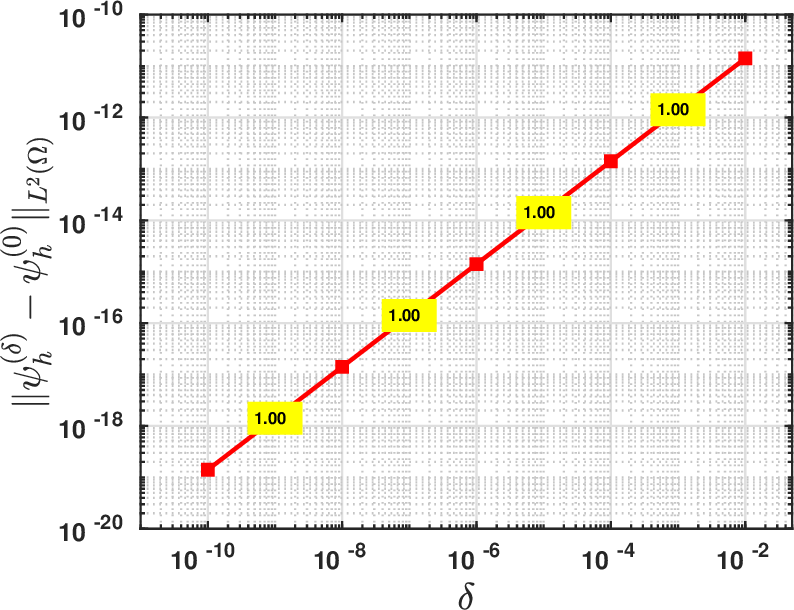} 
    \hspace{0.2in}
    \includegraphics[width = 0.4\textwidth]{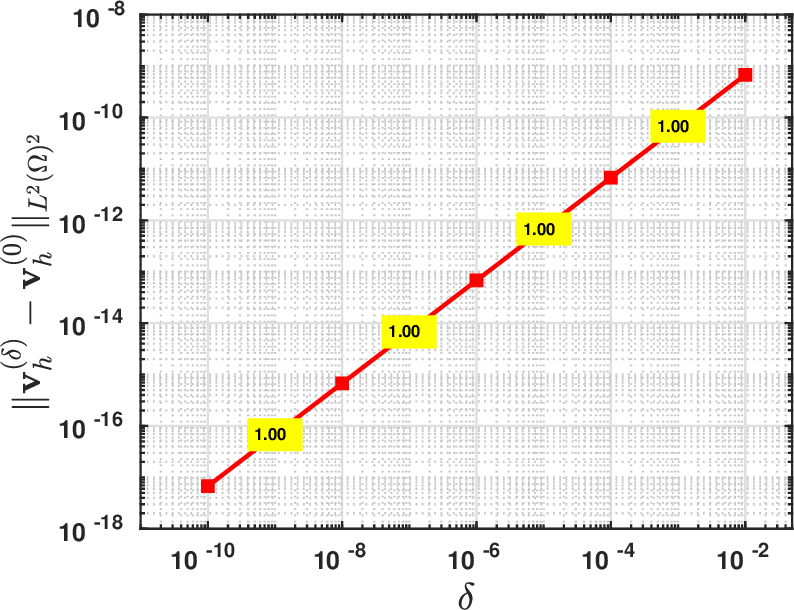} 
    \caption{$\delta$-convergence of the errors in~\eqref{EQN::DELTA-ERRORS} at~$t = 10^{-4}$s for the test case in Section~\ref{SECT::WAVEFRONT} with degree of approximation~$p = 5$.}
    \label{fig:delta-convergence-wavefront}
\end{figure}

\section{Conclusions \label{SECT::CONCLUSIONS}}
In this work, we have designed an asymptotic-preserving HDG method for the numerical simulation of the quasilinear Westervelt equation. 
We built up a well-posedness and approximation theory for this method, and illustrated our theoretical results with two-dimensional numerical experiments.

Optimal~$h$-convergence rates of order~$\mathcal{O}(h^{p+1})$ are proven for the approximation of the acoustic particle velocity~$\bv = \nabla \psi$, thus exceeding the expected convergence rates for most standard DG methods. 
Moreover, we have proven the convergence of the discrete approximation to the vanishing viscosity limit when the sound diffusivity parameter~$\delta$ tends to zero. Such a result guarantees the robustness of the method for small values of~$\delta$.

Our analysis imposes  
a restriction on the degree of approximation of the method, namely~$p\geq 1$. However, in the numerical experiments, we have obtained convergence of the method with respect to~$h$ and~$\delta$ also for~$p=0$. 
This is most likely due to the fact that the numerical experiments were performed for two-dimensional domains. Indeed, the case~$p=0$ is critical for dimension~$d=2$, as commented in Remark~\ref{REM::MINIMUM-DEGREE}.

The following are three interesting possible directions for the extension of our analysis:
\begin{itemize}
\item In view of the close relation between mixed FEM and HDG methods (see, e.g., \cite{Cockburn_Qiu_Shi_2012, Cockburn_Gopalakrishnan_Lazarov_2009}), we expect that the present analysis can be extended to a unified framework covering a large class of methods. 
\item More general polytopic meshes could be considered using the theory 
of~$M$-decompositions~\cite{Cockburn_Fu_Sayas_2017}, or hybrid high-order (HHO) 
methods~\cite{DiPietro_Ern_Lemaire_2014}. In particular, the stabilization term 
used for HHO methods allows for a simpler analysis, in the context of polytopic 
meshes, that does not rely on specific HDG projections.
\item The extension of the method to more general nonlinear sound propagation 
models such as the Kuznetsov equation~\cite{Kuznetsov_1970}.
\end{itemize}

In addition, the superconvergence of order~$\mathcal{O}(h^{p+2})$ for the local postprocessed approximation~$\psih^*$ defined in~\eqref{EQN::POSTPROCESSING}, 
and the asymptotic-preserving properties of fully discrete schemes as 
in~\cite{Dorich_Nikolic_2024}, with special attention to high-order time 
stepping schemes, are the subject of ongoing research.

\section*{Acknowledgements}
We thank the reviewers for their careful reading of the manuscript and helpful 
remarks, which have
led to marked improvements. We are greatful to Vanja Nikoli\'c (Radboud 
University) for her very valuable suggestions to improve the presentation of 
this work.
The first author acknowledges support from the Italian Ministry of
University and Research through the project
PRIN2020 ``Advanced polyhedral discretizations of heterogeneous PDEs for multiphysics problems",
and from the INdAM-GNCS through the project CUP\_E53C23001670001.


\appendix
\section{Proof of low-order energy stability estimate~\texorpdfstring{\eqref{EQN::ENERGY-STABILITY-LINEARIZED-1}}{}}\label{APP:ProofLowerOrder}
The ideas for the proof of the stability estimates are inspired by the~$\delta$-robust analysis carried out in~\cite[\S5]{Meliani_Nikolic_2023} for the mixed {FEM} approximation of the Westervelt equation.

Observe that~\eqref{EQN::SEMI-DISCRETE-LINEARIZED-1}--\eqref{EQN::SEMI-DISCRETE-LINEARIZED-3} imply 
\begin{subequations}
\label{EQN::AUX-LINEARIZED}
\begin{align}
\label{EQN::AUX-LINEARIZED-1-LOW}
c^2 \bm_h\left(\bvht, \brh\right) + c^2b_h\left(\psiht, \brh\right) + c^2e_h\left(\lambdaht, \brh\right)  & = {c^2}(\pdt{1}\bupsilon, \brh)_{\Omega} & &  \forall \brh \in \Qp,\\
\nonumber
m_h((1 + 2k\alpha_h) \psihtt, \wh)  - c^2b_h(\wh, \tbvh)  & \\
\label{EQN::AUX-LINEARIZED-2-LOW}
+ c^2 s_h(\tpsih, \wh) + c^2 f_h(\tlambdah, \wh)  & = (\varphi, \wh)_{\Omega} & & \forall \wh \in \Sp, 
\\
\label{EQN::AUX-LINEARIZED-3-LOW}
-  c^2e_h(\muh, \tbvh) +  c^2f_h(\muh, \tpsih) + c^2 g_h(\tlambdah, \muh) & = 0 & &  \forall \muh \in \Mp.
\end{align}
\end{subequations}

Taking~$\brh = \tbvh$, $\wh = \psiht$, and~$\muh = \lambdaht$ {in~\eqref{EQN::AUX-LINEARIZED}}, and summing the results, we get
\begin{equation}
\label{EQN::LOW-ORDER-IDENTITY}
\begin{split}
m_h((1 & + 2k \alphah) \psihtt, \psiht) \\
& + c^2 \big(\bm_h(\bvht, \bvh) + s_h(\psih, \psiht) + f_h(\lambdah, \psiht) + f_h(\lambdaht, \psih) + g_h(\lambdah, \lambdaht)\big) \\
& + \delta \left(\bm_h(\bvht, \bvht) + s_h(\psiht, \psiht) + 2f_h(\lambdaht, \psiht) + g_h(\lambdaht, \lambdaht)\right) \\ 
& \quad = c^2(\pdt{1}\bupsilon, \bvh)_{\Omega} + \delta(\pdt{1}\bupsilon, \bvht)_{\Omega} + (\varphi, \psiht)_{\Omega} .
\end{split}
\end{equation}

Moreover, the following identities follow from the definition of the discrete bilinear forms~$m_h(\cdot, \cdot)$, $s_h(\cdot, \cdot)$, $f_h(\cdot, \cdot)$, and~$g_h(\cdot, \cdot)$:
\begin{subequations}
\begin{align}
\label{EQN::AUX-IDENTITIES-LOW-1}
& m_h((1 + 2 k \alphah) \psihtt, \psiht) =  \frac12 \frac{d}{dt} \Norm{\sqrt{1 + {2}k \alphah(\cdot, t)} \psiht}{L^2(\Omega)}^2 - m_h(k \alphaht \psiht, \psiht), \\
\nonumber
& \bm_h(\bvht, \bvh)  + s_h(\psih, \psiht) + f_h(\lambdah, \psiht) + f_h(\lambdaht, \psih) + g_h(\lambdah, \lambdaht) \\
\label{EQN::AUX-IDENTITIES-LOW-2}
&\qquad \qquad  = \frac12 \frac{d}{dt} \left(\Norm{\bvh}{L^2(\Omega)^d}^2 + \Norm{\tau^{\frac12} (\lambdah - \psih)}{L^2((\partial \Th)^{\mathcal{I}})}^2 + \Norm{\tau^{\frac12}{\psiht}}{L^2((\partial \Th)^{\calD})}^2 \right), \\
\nonumber 
& \bm_h(\bvht, \bvht) + s_h(\psiht,  \psiht) + 2 f_h(\lambdaht, \psiht) + g_h(\lambdaht, \lambdaht) \\
\label{EQN::AUX-IDENTITIES-LOW-3}
& \qquad \qquad =  \Norm{\bvht}{L^2(\Omega)^d}^2 + \Norm{\tau^{\frac12} (\lambdaht - \psiht)}{L^2((\partial \Th)^{\mathcal{I}})}^2 + \Norm{\tau^{\frac12}\psiht}{L^2((\partial \Th)^{\calD})}^2.
\end{align}
\end{subequations}

Substituting the identities~\eqref{EQN::AUX-IDENTITIES-LOW-1}--\eqref{EQN::AUX-IDENTITIES-LOW-3} into~\eqref{EQN::LOW-ORDER-IDENTITY}, we get
\begin{align}
\nonumber
\frac{d}{dt} \calE_h^{(0)}[\psih, \bvh, \lambdah](t) & + \delta \left(\Norm{\bvht}{L^2(\Omega)^d}^2 + \Norm{\tau^{\frac12} (\lambdaht - \psiht)}{L^2((\partial \Th)^{\mathcal{I}})}^2 + \Norm{\tau^{\frac12}\lambdaht}{L^2((\partial \Th)^{\calD})}^2\right)\\
\label{EQN::IDENTITY-ENERGY0-B-TERMS}
= & \int_{\Omega} k \alphaht {(\psiht)^2} \dx +  c^2(\pdt{1}\bupsilon, \bvh)_{\Omega} + \delta(\pdt{1}\bupsilon, \bvht)_{\Omega} + (\varphi, \psiht)_{\Omega}.
\end{align}
All the terms multiplied by~${\delta}$ on the left-hand side of~\eqref{EQN::IDENTITY-ENERGY0-B-TERMS} are nonnegative. 
By using  
the Cauchy--Schwarz and the Young inequalities, we get
$$\delta(\pdt{1}\bupsilon, \bvht)_{\Omega} 
\leq \frac{\delta}{4} \Norm{\pdt{1}\bupsilon}{L^2(\Omega)^d}^2 + \delta\Norm{\bvht}{L^2(\Omega)^d}^2,$$ 
so the second term cancels out with the one on the left-hand side of~\eqref{EQN::IDENTITY-ENERGY0-B-TERMS}.
Integrating identity~\eqref{EQN::IDENTITY-ENERGY0-B-TERMS}
over~$(0,t)$ and using the H\"older and the Young inequalities, we deduce that
\begin{equation}
\label{EQN::AUX-ENERGY-IDENTITY-4-LOW}
\begin{split}
\calE_h^{(0)}[\psih, \bvh, \lambdah](t) \leq &\ \calE_h^{(0)}[\psih, \bvh, \lambdah](0) \\
& + \left(\Norm{k {(1 + 2k \alphah)^{-1}} \alphaht}{L^1(0, t; L^{\infty}(\Omega))} + \frac{\gamma}{2}\right)\Norm{\sqrt{1 + {2}k \alphah}\psiht}{L^{\infty}(0, t; L^2(\Omega))}^2 \\
& + \frac{1}{2\gamma}\Norm{(1 + {2}k \alphah)^{-\frac12} \varphi}{L^1(0, t; L^2(\Omega))}^2 + {\frac{\delta}{4}}\Norm{\pdt{1}\bupsilon}{L^2(0, t; L^2(\Omega)^d)}^2 \\
&+ {\frac{c^2}{2{\sigma_0}}} \Norm{\pdt{1}\bupsilon}{L^1(0, t; L^2(\Omega)^d)}^2 + {\frac {c^2 {\sigma_0}}2 }\Norm{{\bvh}}{L^\infty(0,t;L^2(\Omega)^d)}^2,
\end{split}
\end{equation}
for all~$\gamma > 0$.

Moreover, by using the H\"older inequality, we have the following bounds:
\begin{equation}
\label{EQN::AUX-ENERGY-IDENTITY-6}
\begin{split}
\Norm{(1 + 2k \alphah)^{-\frac12} \varphi}{L^1(0, t; L^2(\Omega))}^2 & \leq t \Norm{(1 + 2k \alphah)^{-\frac12} \varphi}{L^2(0, t; L^2(\Omega))}^2,\\
\Norm{\dpt \bupsilon}{L^1(0, t; L^2(\Omega)^d)}^2 & \leq t \Norm{\dpt \bupsilon}{L^2(0, t; L^2(\Omega)^d)}^2, \\
\Norm{\bvh}{L^{\infty}(0, t; L^2(\Omega)^d)}^2 & \leq \sup_{s \in (0, t)} \Norm{\bvh}{L^2(\Omega)^d}^2.
\end{split}
\end{equation}

Finally, the smallness assumption~\eqref{EQN::ENERGY-ASSUMPTION} {states that} there exist constants~$0 < \gamma_0 < {\sigma_0} < 1$ independent of~$h$ and~${\delta}$ such that
\begin{equation*}
\Norm{k {(1 + 2k \alphah)^{-1}} \alphaht}{L^1(0, t; L^{\infty}(\Omega))} + \frac{\gamma_0}{2} \leq \frac{|k|}{1 - 2|k|\ulal} \Norm{\alphaht}{L^1(0, t; L^\infty(\Omega))} + \frac{\gamma_0}{2}\leq \frac{{\sigma_0}}{2},
\end{equation*}
which, together with~\eqref{EQN::AUX-ENERGY-IDENTITY-4-LOW} {and~\eqref{EQN::AUX-ENERGY-IDENTITY-6}}, gives the low-order energy estimate {in}~\eqref{EQN::ENERGY-STABILITY-LINEARIZED-1}.

\section{Proof of {high}-order energy stability estimate~\texorpdfstring{\eqref{EQN::ENERGY-STABILITY-LINEARIZED-2}}{}}\label{APP:ProofHigherOrder}
The proof of the 
{high}-order stability estimate {in}~\eqref{EQN::ENERGY-STABILITY-LINEARIZED-2} follows by considering the time-differentiated system
\begin{align*}
c^2 \bm_h\left(\bvhtt, \brh\right) + c^2b_h\left(\psihtt, \brh\right) + c^2e_h\left(\lambdahtt, \brh\right)  & =  (\pdt{2}\bupsilon, \brh)_{\Omega}  & & \forall \brh \in \Qp,\\
\nonumber
m_h((1 + 2k\alpha_h) \psihttt, \wh) + m_h(2k\pdt{1}{\alpha}_h \psihtt, \wh)   - c^2b_h(\wh, \tbvht)  & \\
+ c^2 f_h(\tlambdaht, \wh) + c^2 s_h(\tpsiht, \wh) & = (\pdt{1}{\varphi}, \wh)_{\Omega} & & \forall \wh \in \Sp, 
\\
-  c^2e_h(\muh, \tbvht) +  c^2f_h(\muh, \tpsiht) + c^2 g_h(\tlambdaht, \muh) & = 0 & &  \forall \muh \in \Mp,
\end{align*}
choosing
$\brh = \tbvht$, $\wh = \psihtt$, and~$\muh = \lambdahtt$ as test functions,
and summing the resulting equations. The remaining steps are similar to those exposed in Appendix~\ref{APP:ProofLowerOrder} for the low-order estimate in~\eqref{EQN::ENERGY-STABILITY-LINEARIZED-1}, and are therefore omitted. 
\end{document}